\documentclass[11pt,a4paper]{article}
\usepackage[cp1251]{inputenc}
\usepackage{amsmath,amsthm}
\usepackage{amsfonts,amstext,amssymb,verbatim,epsfig}
\usepackage{dsfont}
\usepackage{enumerate}
\usepackage{psfrag}
\usepackage{graphicx}

\def\R{{\mathbb{R}}}
\def\N{{\mathbb{N}}}
\def\Z{{\mathbb{Z}}}

\newcommand{\GG}{\mathbb{G}}
\renewcommand{\P}{\mathbb{P}}

\newcommand{\ballZ}{\mathrm{B}} 

\def\badseed{{\overline G}}

\def \p{{\bf P1}} 
\def \pp{{\bf P2}} 
\def \ppp{{\bf P3}} 
\def \s{{\bf S1}} 
\def \sss{{\bf S2}}

\def \atom{\xi} 
\def \set{{\mathcal S}}
\def \distS{{\rho_{\scriptscriptstyle \set}}}
\def \ballS{{\ballZ_{\scriptscriptstyle \set}}}

\def \ballSZ{{\widetilde \ballZ_{\scriptscriptstyle \set_\infty}}}
\def \distSZ{{\widetilde \rho_{\scriptscriptstyle \set_\infty}}}
\def \constP{{\chi_{\scriptscriptstyle {\mathrm P}}}}
\def \constS{{\Delta_{\scriptscriptstyle {\mathrm S}}}}
\def \funcP{f_{\scriptscriptstyle {\mathrm P}}}
\def \funcS{f_{\scriptscriptstyle {\mathrm S}}}
\def \RP{R_{\scriptscriptstyle {\mathrm P}}}
\def \RS{R_{\scriptscriptstyle {\mathrm S}}}
\def \LP{L_{\scriptscriptstyle {\mathrm P}}}
\def \epsP{{\varepsilon_{\scriptscriptstyle {\mathrm P}}}}

\def \scexp{{\theta_{\scriptscriptstyle {\mathrm sc}}}}

\newcommand{\chemdist}[1]{\rho_{#1}}
\def \chemconstt{{\mathrm C}_{\scriptscriptstyle \rho}} 
\def \shaperad{R} 

\newcommand{\ballR}{ \mathrm{Q} } 

\def \binlog{\log} 

\newtheorem{theorem}{Theorem}[section]
\newtheorem{corollary}[theorem]{Corollary}
\newtheorem{lemma}[theorem]{Lemma}
\newtheorem{proposition}[theorem]{Proposition}

\newtheoremstyle{likedef}
  {}%
  {}%
  {}%
  {\parindent}%
  {\bfseries}%
  {.}%
  {.5em}%
  {}%

\theoremstyle{likedef}
\newtheorem{definition}[theorem]{Definition}
\newtheorem{remark}[theorem]{Remark}

\numberwithin{equation}{section}

\begin{document}
\title{On chemical distances and shape theorems in percolation models with long-range correlations}

\author{
Alexander Drewitz
\thanks{Columbia University, Department of Mathematics, RM 614, MC 4419,
2990 Broadway, New York City, NY 10027, USA. email: drewitz@math.columbia.edu}
\and
Bal\'azs R\'ath
\thanks{
The University of British Columbia, Department of Mathematics,
Room 121, 1984 Mathematics Road, Vancouver, B.C.
Canada V6T 1Z2. email: rathb@math.ubc.ca}
\and
Art\"{e}m Sapozhnikov
\thanks{
University of Leipzig, Department of Mathematics, Room A409, 
Augustusplatz 10, 04109 Leipzig, Germany.
email: artem.sapozhnikov@math.uni-leipzig.de}
}

\maketitle


\begin{abstract}
In this paper we provide general conditions on a one parameter family of random infinite subsets of $\Z^d$ 
to contain a unique infinite connected component for which the chemical distances are comparable to the Euclidean distance.
In addition, we show that these conditions also imply a shape theorem for the corresponding infinite connected component.

By verifying these conditions for specific models, we obtain novel results about
 the structure of the infinite 
connected component of the vacant set of random interlacements and the 
level sets of the Gaussian free field. 
As a byproduct we  obtain alternative proofs to the corresponding results for random interlacements
in \cite{CP}, and while our main interest
is in percolation models with long-range correlations, we also
recover results in the spirit of \cite{antal_pisztora} for Bernoulli percolation.

Finally, as a corollary, we derive new results about the (chemical) diameter of the largest
connected component in the complement of the trace of the random walk on the torus. 
\end{abstract}

\bigskip
\bigskip
\bigskip

Percolation was introduced in \cite{broadbent_hammersley} as a mathematical model of a porous medium, 
where the physical space is modeled by the lattice $\Z^d$, and the pure substance is described as a random 
subset $\set$ of $\Z^d$. 
A central question is to understand physical properties of $\set$, and how universal these properties are for various distributions of $\set$. 
Mathematically speaking, what can be said about the connectivity properties and the geometry of the subgraph of $\Z^d$ spanned by $\set$?
In this paper we consider the general
situation when the set $\set$ has a unique infinite connected component $\set_\infty$ 
which covers a positive fraction of $\Z^d$. 
We define the \emph{chemical distance} of $x,y \in \set_\infty$ to be the length of the shortest nearest neighbor path connecting $x$ and $y$ in $\set_\infty$, 
and we want to understand when the long-scale behavior of the resulting random metric space is similar to that of the underlying space $\Z^d$.

In the case of supercritical Bernoulli percolation (when every vertex of $\Z^d$ is in $\set$ with probability $p>p_c$ independently of each other), 
\cite[Theorem 1.1]{antal_pisztora} asserts the existence of a constant $C=C(d,p)$ such that 
the probability that the chemical distance of  $x, y \in \set_\infty$ is greater than $C  |x-y|$ decays 
exponentially
as $|x-y| \to \infty$. 
Using this control on chemical distances as well as Kingman's subadditive ergodic theorem one can deduce a \emph{shape theorem} 
which states that the chemical ball of radius $R$ in $\set_\infty$, rescaled by $R$, converges almost surely with respect to the Hausdorff distance 
to a deterministic compact convex set of $\R^d$ as $R \to \infty$ (see, e.g., \cite[Corollary 5.4]{garet_marchand}).

 In this paper we extend the shape theorem from the supercritical
 Bernoulli setting to a general class of correlated percolation models, by
 proving novel results about chemical distances in these models. In our
 setup, we deal with a one-parameter family of probability measures $\mathbb
 P^u$ (describing the law of $\set \subseteq \Z^d$ at different densities)
 satisfying the following conditions:

\begin{itemize}
\item \p{}: spatial ergodicity;

\item \pp{}: stochastic monotonicity in $u$;

\item \ppp{}: weak decorrelation for monotone events;

\item \s{}: local uniqueness of $\set_\infty$;

\item \sss{}: continuity of the density of $\set_\infty$ in $u$;

\end{itemize}

see Section \ref{sec:modRes} for the rigorous definitions and
 our general results
Theorem \ref{thm:Sinfty:chemdist} 
(which states that chemical distance on $\set_\infty$ is comparable to Eucledian distance)
 and Theorem \ref{thm:shapeThm} (the shape theorem).

As we will show, the above conditions are general enough to be satisfied by many models, including 
random interlacements, the vacant set of random interlacements, and the level sets of the Gaussian free field (see Sections~\ref{sec:ri}, \ref{sec:vsri}, and \ref{sec:gff}, respectively). 
In our context, the last two models are particularly challenging since, in contrast to random interlacements,
 they exhibit percolation phase transitions. For the time being we cannot deal with their 
entire supercritical phases. However, our framework allows us to identify the only condition that is missing:
 as
soon as \s{} is extended up to criticality, our results will also hold for the entire supercritical phases. 
Let us also emphasize here that the above 
models give rise to random subsets of $\Z^d$ with {\it polynomial decay 
of correlations}, but still they satisfy the weak decorrelation 
condition \ppp{}, as we show in Sections~\ref{sec:ri}, \ref{sec:vsri}, and \ref{sec:gff}. 

\medskip

Historically, the asymptotic linear behaviour of chemical distances along rays in supercritical
 Bernoulli percolation was proved for all $d \geq 2$ and  $p>p_c$ 
  in \cite[(5.5)]{GM90}. In order to deduce the shape theorem, one also needs
 a uniform control of the chemical distances in large balls: 
 the exponentially decaying bound (for all $p>p_c$) of \cite[Theorem 1.1]{antal_pisztora} improved 
 the polynomially decaying bound (for sufficiently high $p$) of \cite[Lemma 2.8]{Gartner_Molchanov}.
Corresponding properties of the chemical distance on the random interlacement $\mathcal{I}^u$ at level $u$
 have recently been derived in \cite[Theorems 1.1 and 1.3]{CP}.
However, the techniques of \cite{CP} heavily 
rely on the specific connectivity properties of  $\mathcal{I}^u$,
and in particular, the question of whether these results hold for the vacant set 
$\mathcal{V}^u= \Z^d \setminus \mathcal{I}^u$  of random interlacements also
had remained unsolved (and similarly for the level sets of Gaussian free field).
In addition to answering these questions positively, our method also
provides a general, model independent approach. This allows for a possible application our results
to other percolation models by checking the conditions \p{} -- \ppp{} and \s{} -- \sss{}
in order to derive chemical distance results and shape theorems for them.
In this context, it would for example be interesting to see whether the random cluster model
(see \cite{GrimmettCluster} for a reference) satisfies all of the above conditions.

Another potential advantage of our general framework is that it seems to capture the key
 properties (i.e., \p{} -- \ppp{} and \s{} -- \sss{})  of a dependent percolation model which imply that the geometry of $\set_\infty$ is similar to that of $\Z^d$. In particular, recent progress
indicates that our conditions imply that simple random walk on $\set_\infty$ behaves similarly to random walk
 on $\Z^d$, see Remark \ref{remark_isoperimetry} for further discussion.

Finally, using a connection between random interlacements and random walk on the discrete torus, 
we obtain new results about the (chemical) diameter of the giant connected component in the complement of the random walk trace (see Section~\ref{sec:rwtorus}).

\section{Model and results} \label{sec:modRes}

We consider a one parameter family of probability measures $\mathbb P^u$, $u\in(a,b)\subseteq \R_+$, on the measurable space $(\{0,1\}^{\Z^d},\mathcal F)$, $d\geq 2$, 
where the sigma-algebra $\mathcal F$ is generated by the canonical coordinate maps $\Psi_x : \{0,1\}^{\Z^d}\to \{0,1\}$, $x\in\Z^d$ 
(i.e., $\Psi_x(\atom) = \atom_x$ for $\atom\in\{0,1\}^{\Z^d}$ and $x\in\Z^d$).

For any $\atom\in\{0,1\}^{\Z^d}$, we define 
\[
\set = \set(\atom) = \{x\in\Z^d~:~\xi_x = 1\} \subseteq \Z^d .\
\]
We view $\set$ as a subgraph of $\Z^d$ in which the edges are drawn between any two vertices of $\set$ within $\ell^1$-distance $1$ from each other. 
(For $x=(x^{(1)},\dots,x^{(d)})\in \R^d$, the $\ell^1$-norm of $x$ is defined in the usual way by $|x|_1 = \sum_{i=1}^d|x^{(i)}|$.)
We denote by $\set_\infty$ the subset of vertices of $\set$ which are in infinite connected components of $\set$. 

In this paper we will focus on connectivity properties of $\set_\infty$. 
We will prove that under certain conditions (see \p{} -- \ppp{} and \s{} -- \sss{}) on the family of probability measures $\mathbb P^u$, $u\in(a,b)$, 
the graph $\set_\infty$ is connected and ``looks like'' the underlying lattice $\Z^d$.
Roughly speaking, we show that on large scales the graph distance in $\set_\infty$ behaves like a metric induced by a norm on $\R^d$.

We will now define conditions on the family $\mathbb P^u$, $u\in(a,b)$. 
After that we will state the main results (Theorems \ref{thm:Sinfty:chemdist} and \ref{thm:shapeThm}) of the paper. 
Particular examples of probability measures satisfying the conditions (and results) below will be given in Section~\ref{sec:examples}. Note that our aim was to formulate our conditions in a general and flexible form, in order to facilitate the extension of the main results of this paper to other percolation models.

The numbers $0\leq a<b$ as well as the dimension $d\geq 2$ are going to be fixed throughout the paper, 
and we omit the dependence of various constants on $a$, $b$, and $d$. 
\begin{itemize}
\item[\p{}]
For any $u\in(a,b)$, $\mathbb P^u$ is invariant and ergodic with respect to the lattice shifts.
\end{itemize}
To state the next two conditions we need the following definition. 
\begin{equation*}
\begin{array}{c}
\text{An event $G \in \mathcal F$ is called \emph{increasing} (respectively, \emph{decreasing}), if}\\ 
\text{for all $\atom\in G$ and $\atom' \in \{0,1\}^{\Z^d}$ with $\atom_y \leq \atom_y'$ (respectively, $\atom_y \geq \atom_y'$) for all $y\in\Z^d$, one has $\atom' \in G$.}
\end{array}
\end{equation*}
\begin{itemize}
\item[\pp{}]
For any $u,u'\in(a,b)$ with $u<u'$, and any increasing event $G\in\mathcal F$, 
$\mathbb P^u[G] \leq \mathbb P^{u'}[G]$.
(Usually, this condition is referred to as stochastic monotonicity of $\mathbb P^u$.)
\end{itemize}
Conditions \p{} and \pp{} are rather general and are satisfied by many families of probability measures. 
The condition \ppp{} below is more restrictive than the above two. It states that $\mathbb P^u$, $u\in(a,b)$, satisfy a certain weak decorrelation inequality. 
For $x \in \Z^d$ and $r \in \R_+$, we denote by 
\[\ballZ(x,r) = \{y\in\Z^d~:~|x-y|_\infty\leq \lfloor r \rfloor \}\] the closed $l^{\infty}$-ball in $\Z^d$ with 
radius $\lfloor r \rfloor$ and center $x$. 
(For $x=(x^{(1)},\dots,x^{(d)})\in \R^d$, the $\ell^\infty$-norm of $x$ is defined by $|x| = |x|_\infty = \max\{|x^{(1)}|,\ldots|x^{(d)}|\}$.)
We write $\ballZ(r)$ for $\ballZ(0,r)$. 

\medskip

\begin{itemize}
\item[\ppp{}]
Let $L\geq 1$ be an integer. Let $x_1,x_2\in\Z^d$. 
For $i\in\{1,2\}$, let $A_i\in\sigma(\Psi_y~:~y\in \ballZ(x_i,10L))$ be decreasing events, and 
$B_i\in\sigma(\Psi_y~:~y\in \ballZ(x_i,10L))$ increasing events. 
There exist $\RP,\LP <\infty$ and $\epsP,\constP>0$ such that for any integer $R\geq \RP$ and $a<\widehat u<u<b$ satisfying 
\begin{equation}\label{eq:uwidehatu}
u\geq \left(1 + R^{-\constP}\right)\cdot \widehat u ,\
\end{equation}
if $|x_1 - x_2|_\infty \geq R\cdot L$, then 
\begin{equation}\label{eq:decorrelation:decreasing}
\mathbb P^u\left[A_1\cap A_2\right] \leq 
\mathbb P^{\widehat u}\left[A_1\right] \cdot
\mathbb P^{\widehat u}\left[A_2\right] 
+ e^{-\funcP(L)} ,\
\end{equation}
and
\begin{equation}\label{eq:decorrelation:increasing}
\mathbb P^{\widehat u}\left[B_1\cap B_2\right] \leq 
\mathbb P^u\left[B_1\right] \cdot
\mathbb P^u\left[B_2\right] 
+ e^{-\funcP(L)} ,\
\end{equation}
where $\funcP$ is a real valued function satisfying
\begin{equation}\label{def:epsP}
\text{$\funcP(L) \geq e^{(\log L)^\epsP}$ for all $L\geq \LP$.}
\end{equation}
\end{itemize}
\begin{remark}\label{rem:decorrelation}
At first sight it may look like condition \ppp{} is rather strong and asserts that the correlations should decay 
much faster than any polynomial ($e^{\funcP(L)} \gg L^k$ for any $k$). 
However, the fact that we deal with different parameters $u$ and $\widehat u$ (and a very restricted class of events) allows for measures with polynomial decay of correlations to satisfy \ppp{}. 
Examples of such families of measures for $d\geq 3$ are random interlacements and the level sets of the Gaussian free field, for which  
\[
\left|\mathbb P^u[\Psi_x = 1, \Psi_y = 1] - \mathbb P^u[\Psi_x = 1]\mathbb P^u[\Psi_y = 1]\right| \asymp (1+|x-y|)^{2-d}, \qquad x,y\in\Z^d .\
\]
We will discuss these examples in more details in Sections~\ref{sec:ri} -- \ref{sec:gff}. 
In the literature, the idea behind the proof of inequalities of the form \eqref{eq:decorrelation:decreasing} and \eqref{eq:decorrelation:increasing} is often referred to as ``sprinkling''.
\end{remark}

\bigskip

While \p{} -- \ppp{} are general conditions on the family $\mathbb P^u$, $u\in(a,b)$, the next condition \s{} pertains to the connectivity properties of $\set$. 
It can be understood as a certain local uniqueness property for macroscopic connected components in $\set$. 
Roughly speaking, this condition implies that with high probability, large enough boxes intersect exactly one connected component of $\set$ with large diameter.

\begin{equation}\label{def:setr}
\begin{array}{c}
\text{For $r\in [0,\infty]$, we denote by $\set_r$, the set of vertices of $\set$}\\
\text{which are in connected components of $\set$ of ($\ell^1$-)diameter $\geq r$.}
\end{array}
\end{equation} 
In particular, in the case $r=\infty$ we recover the set $\set_\infty$ of vertices of $\set$ contained in infinite connected components of $\set$.

\medskip

\begin{itemize}
\item[\s{}] 
There exists a function $\funcS:(a,b)\times\Z_+\to \mathbb R$ such that 
\begin{equation}\label{eq:funcS}
\begin{array}{c}
\text{for each $u\in(a,b)$, there exist $\constS = \constS(u)>0$ and $\RS = \RS(u)<\infty$}\\
\text{such that $\funcS(u,R) \geq (\log R)^{1+\constS}$ for all $R\geq \RS$,}
\end{array}
\end{equation}
and for all $u\in(a,b)$ and $R\geq 1$, the following inequalities are satisfied:
\begin{equation*} 
\mathbb P^u\left[ \, 
\set_R\cap\ballZ(R) \neq \emptyset \,
\right]
\geq 
1 - e^{-\funcS(u,R)} , 
\end{equation*}
and
\begin{equation} \label{eq:locUnProb}
\mathbb P^u\left[
\bigcap_{x,y\in\set_{R/10}\cap\ballZ(R)}
\left\{ \,
\text{$x$ is connected to $y$ in }
\set\cap\ballZ(2R) \, 
\right\}
\right]
\geq 1 - e^{-\funcS(u,R)} . 
\end{equation}
\end{itemize}

\medskip

It is not difficult to see that if the family $\mathbb P^u$, $u\in(a,b)$, satisfies \s{}, then for any $u\in(a,b)$, 
\begin{equation}\label{eq:Sinfty}
\text{
$\mathbb P^u$-a.s., the set $\set_\infty$ is non-empty and connected,
}
\end{equation}
and there exist $c_{\scriptscriptstyle 1}=c_{\scriptscriptstyle 1}(u)>0$ and $C_1 = C_1(u)<\infty$ such that for all $R\geq 1$, 
\begin{equation}\label{eq:C1:infty}
\mathbb P^u\left[ \,
\set_\infty\cap\ballZ(R) \neq \emptyset \, 
\right]
\geq 
1 - C_1 e^{-c_{\scriptscriptstyle 1}\funcS(u,R)} . 
\end{equation}
Moreover, using a standard covering argument, if the family $\mathbb P^u$, $u\in(a,b)$, is invariant under the lattice shifts and satisfies \s{}, then 
for any $u\in(a,b)$ and $\varepsilon>0$, 
there exist $c_{\scriptscriptstyle 2}=c_{\scriptscriptstyle 2}(u,\varepsilon)>0$ and $C_2 = C_2(u,\varepsilon)<\infty$ such that 
for all $R\geq 1$, 
\begin{equation}\label{eq:C2:epsilon}
\mathbb P^u\left[
\bigcap_{x,y\in\set_{\varepsilon R}\cap\ballZ(R)}
\left\{ \,
\text{$x$ is connected to $y$ in }\set\cap\ballZ((1+\varepsilon)R) \,
\right\}
\right]
\geq 1 - C_2 e^{-c_{\scriptscriptstyle 2}\funcS(u,R)} .\
\end{equation}

We omit the proof of \eqref{eq:C2:epsilon} and
 refer the reader to the derivation of Proposition~1 from Lemma~13 in \cite{RS:Transience}, where essentially the same statement is proved.

\bigskip

Our final condition \sss{} on the measures $\mathbb P^u$, $u\in(a,b)$, concerns the density of $\set_\infty$, 
\begin{equation}\label{def:etau}
\eta(u) = \mathbb P^u\left[0\in\set_\infty\right] .\
\end{equation}

\medskip

\begin{itemize}
\item[\sss{}]
The function $\eta(\cdot)$ is positive and continuous on $(a,b)$. 
\end{itemize}
In fact, the positivity of $\eta(\cdot)$ already follows from \eqref{eq:Sinfty}, if we assume that $\mathbb P^u$, $u\in(a,b)$, are invariant with respect to lattice shifts and satisfy \s{}. 
Note that if we assume \pp{}, then 
\begin{equation*}
\text{$\eta(\cdot)$ is non-decreasing on $(a,b)$.}
\end{equation*}
\begin{remark}
Condition \sss{} is the price that we pay for asking for decorrelation inequalities only in the weak form of \ppp{} (see the proofs of Lemmas \ref{l:Auxk} and \ref{l:Buxk}). 
If the inequalities \eqref{eq:decorrelation:decreasing} and \eqref{eq:decorrelation:increasing} in \ppp{} were true for $\widehat u = u$, then we would not need the assumption \sss{}. 
However, as we discuss in Remark~\ref{rem:decorrelation}, there are many percolation models with long-range correlations for which 
\eqref{eq:decorrelation:decreasing} and \eqref{eq:decorrelation:increasing} fail to hold with $\widehat u = u$, but 
conditions \ppp{} and \sss{} are still valid. 
\end{remark}

The main result of our paper is the following theorem. 
For $x,y\in\set$, let $\distS(x,y)\in\Z_+\cup\{\infty\}$ denote the {\it chemical distance} (also known as the internal or graph distance) in $\set$ between $x$ and $y$, i.e., 
\begin{equation*}
\distS(x,y) = \inf \left\{ \, n \geq 0 \; : \; 
\begin{array}{c}
\text{there exist $x_0, \dots, x_n\in \set$ such that $x_0 = x$, $x_n = y$,}\\ 
\text{and $|x_k - x_{k-1}|_1=1$ for all $k=1,\dots, n$}
\end{array}
\right\}.
\end{equation*}
Here we use the usual convention $\inf \emptyset  = \infty$, i.e., we set $\distS(x,y) = \infty$ if 
$x$ and $y$ are in different connected components of $\set$.

\begin{theorem}[Chemical distance]\label{thm:Sinfty:chemdist}
Assume that the family of probability measures $\mathbb P^u$, $u\in(a,b)$, on $(\{0,1\}^{\Z^d},\mathcal F)$, $d\geq 2$, satisfies 
the conditions \p{} -- \ppp{} and \s{} -- \sss{}.
For any $u\in(a,b)$, there exist $c = c(u)>0$ and $C = C(u)<\infty$ such that for all $R\geq 1$, 
\begin{equation}\label{eq:Sinfty:chemdist}
\mathbb P^u\left[
\text{ for all $x,y\in\set_R\cap \ballZ(R)$, $\distS(x,y) \leq C R$ }
\right]\geq 1 - Ce^{-c(\binlog R)^{1+\constS}} ,\
\end{equation}
where $\constS$ comes from condition \s{} (see \eqref{eq:funcS}).
\end{theorem}
\begin{remark}
If a probability measure $\mathbb P^u$ is invariant under the lattice shifts and satisfies \s{}, then 
using a union bound one can already derive a weak version of \eqref{eq:Sinfty:chemdist}. 
Indeed, for any $\varepsilon>0$, there exist $c = c(u,\varepsilon)>0$ and $C = C(u,\varepsilon)<\infty$ such that for all $R\geq 1$, 
\[
\mathbb P^u\left[
\text{ for all $x,y\in\set_R\cap \ballZ(R)$, $\distS(x,y) \leq C R^{1+\varepsilon}$ }
\right]\geq 1 - Ce^{-c(\binlog R)^{1+\constS}} .\
\]
However, in order to prove \eqref{eq:Sinfty:chemdist}, we need to quantify correlations in some way. 
In this paper we achieve this by assuming \ppp{}. 
\end{remark}

\medskip

Using standard methods, we can deduce from Theorem~\ref{thm:Sinfty:chemdist} a shape theorem for $\distS$-balls of $\set_\infty$. 
For $x\in\set$ and $r \ge 0$, we denote by $\ballS(x,r)$ the ball in $\set$ with center $x$ and radius $\lfloor r\rfloor$, i.e., 
\begin{equation*}
\ballS(x,r) = \left\{y\in\set~:~\distS(x,y)\leq r \right\} .\
\end{equation*}
The shape theorem states that large balls in $\set_\infty$ with respect to the metric $\distS$ after rescaling 
have an asymptotic deterministic shape. 

\begin{theorem} [Shape theorem] \label{thm:shapeThm}
Assume that the family of probability measures $\mathbb P^u$, $u\in(a,b)$, satisfies \p{} -- \ppp{} and \s{} -- \sss{}. 
Then for any $u\in(a,b)$ there exists a convex compact set $D_u = D_u(\mathbb P^u) \subset \R^d$
such that for each $\varepsilon \in (0,1)$, 
there exists a $\P^u [ \; \cdot \; \vert \, 0 \in \set_\infty]$-almost surely finite 
random variable $\tilde{\shaperad}_{\varepsilon,u}$ satisfying
\begin{equation*}
\forall \,
\shaperad \geq \tilde{\shaperad}_{\varepsilon,u} \; : \; 
\set_\infty \cap (1-\varepsilon)\shaperad \cdot D_u 
\subseteq \ballS (0,\shaperad) \subseteq
\set_\infty  \cap   (1+\varepsilon) \shaperad \cdot D_u .
\end{equation*}
\end{theorem}
\begin{remark} 
(i) The set $D_u$ preserves symmetries of $\mathbb P^u.$ In particular, if 
$\mathbb P^u$ is invariant with respect to the isometries of $\R^d$ which preserve $0$ and $\Z^d$, then so is $D_u$. \\
(ii) Since $\mathbb P^u$, $u\in(a,b)$, satisfy \pp{}, it follows that 
for any $u,u'\in(a,b)$ with $u<u'$, $D_u \subseteq D_{u'}$.
\end{remark}

We prove Theorem \ref{thm:shapeThm} in Section \ref{section:shape_proof} using Theorem \ref{thm:Sinfty:chemdist} and
Kingman's subadditive ergodic theorem \cite{kingman} in a standard fashion.

\bigskip

Let us now describe the strategy for the proof of Theorem~\ref{thm:Sinfty:chemdist}. 
In this description we will say that a path is ``short" if its length is comparable to the $\ell^1$-distance of its
 endvertices. 
We will set up a general multi-scale renormalization scheme, which
involves the recursive definition of unfavorable regions on higher and higher scales. 
Our aim is to 
iteratively construct short nearest neighbor paths in $\set$: given a short path $\pi$ (not necessarily in $\set$) which avoids the unfavorable regions on a given scale, 
we modify this path in a way that the resulting path $\pi'$ (still not necessarily in $\set$) avoids the unfavorable regions on the previous scale as well. 
We choose our scales to grow faster than geometrically in order to achieve an upper bound on the length of the iteratively constructed path which is uniform in the number of iterations. 
Our definition of the favorable region on the bottom scale implies (due to \s{}) that such region
contains a 
unique macroscopic connected component which is locally connected to the macroscopic connected components of all the neighboring favorable regions. 
By ``glueing'' such connected components together, we are able to locally modify 
short paths (not necessarily in $\set$) that avoid unfavorable regions on the bottom scale 
to obtain short nearest neighbor paths in $\set$. 
The weak decorrelation inequalities of \ppp{} allow us to set up a renormalization scheme in which unfavorable regions on higher scales become very unlikely. 
This is a tricky part, since we should be satisfied only with monotone events 
(note here that the local uniqueness event in the probability of inequality \eqref{eq:locUnProb} is not monotone!)
and cope with different parameters ($u$ and $\widehat u$) in the decorrelation inequalities.  
All in all, favorable regions on high scales are likely and allow for short paths in $\set$. 
We use this conclusion to define an event $\mathcal H$ such that 
(a) $\mathcal H$ implies the event on the left-hand side of \eqref{eq:Sinfty:chemdist} and 
(b) $\mathbb P^u[\mathcal H]$ is at least $1 - Ce^{-c(\binlog R)^{1+\constS}}$.

\begin{remark}\label{remark_isoperimetry}
It is natural to ask whether the conditions \p{} -- \ppp{} and \s{} -- \sss{} are sufficient for obtaining finer results about the structure of $\set_\infty$, 
e.g., quenched {\it Gaussian bounds on the transition kernel} of a simple random walk on $\set_\infty$ and a quenched {\it invariance principle} for the walk. 
In the case of Bernoulli percolation, the corresponding Gaussian bounds were proved in \cite{Barlow} (see also \cite{MathieuRemy} for a partial result) 
using extensively techniques and results of \cite{antal_pisztora}, and the quenched invariance principle was subsequently proved in \cite{SidoraviciusSznitman_RW} for $d\geq 4$ and 
in \cite{BergerBiskup,MathieuPiatnitski} for all $d\geq 2$ using the results of \cite{Barlow}. 
Techniques of the above papers heavily rely on tools valid only for weakly dependent models, such as the Liggett-Schonmann-Stacey theorem \cite{LSS}, 
and are not applicable even to the specific models treated in Sections~\ref{sec:ri}--\ref{sec:gff}. 
We believe that the multi-scale renormalization scheme developed in this paper will find applications in the proof of the above questions not only for specific examples of models, but in the full 
generality of assumptions \p{} -- \ppp{} and \s{} -- \sss{}. 
In fact, the quenched invariance principle for random walk on $\set_\infty$ under assumptions \p{} -- \ppp{} and \s{} -- \sss{} has been recently proved in \cite{PRS_isoperimetric}. 
The proof of \cite{PRS_isoperimetric} is based on a careful analysis of our renormalization scheme and our main result Theorem~\ref{thm:Sinfty:chemdist}. 
Also, the recent preprint \cite{Okamura} takes advantage of our setup and main result
Theorem~\ref{thm:Sinfty:chemdist} in order to derive quenched large deviations for simple random walk
on the infinite connected component $\set_\infty.$ 
\end{remark}

\bigskip 

The rest of the paper is organized as follows. 

In Section~\ref{sec:examples} we give applications of our main results by demonstrating particular models in which the conditions \p{} -- \ppp{} and \s{} -- \sss{} are satisfied.
Our main focus is on the models for which {\it only} weak decorrelation inequalities of the form \eqref{eq:decorrelation:decreasing} and \eqref{eq:decorrelation:increasing} are available. 
In particular, in Section~\ref{sec:ri} we show that our results give an alternative approach to the results of \cite{CP} for random interlacements, 
in Section~\ref{sec:vsri} we show that our results hold for the vacant set of random interlacements in a (sub)phase of the supercritical parameters, 
and in Section~\ref{sec:gff} we show that the level sets of the Gaussian free field fit into our setup for a (sub)phase of the supercritical parameters. 
The results of Sections~\ref{sec:vsri} and \ref{sec:gff} are new. 
Finally, in Section~\ref{sec:rwtorus} we apply results of Section~\ref{sec:vsri} to obtain new results about the (chemical) diameter of the giant connected component 
in the complement of the random walk trace on the torus. 

In Section~\ref{section:renorm} we define the multi-scale renormalization scheme in an abstract setting. 
We define the notion of unfavorable regions and state conditions under which the higher scale unfavorable regions are unlikely. 

In Section \ref{section:conn_patterns} we define the connectivity patterns of favorable regions on the bottom scale of our scheme, 
bringing into play two families of events of high probability, one increasing and one decreasing. 

In Section \ref{section:good_and_bad} we use the events from Section~\ref{section:conn_patterns} and the notion of favorable/unfavorable regions from Section~\ref{section:renorm} 
to define good and bad regions on all scales. We then describe explicitly how to modify a path through a good region on one scale to obtain a path through a good region on a 
lower scale without increasing its length by too much. This gives us one iteration in the construction of short paths in $\set$. 

In Section \ref{section:proof_of_chemdist_thm} we prove Theorem \ref{thm:Sinfty:chemdist}. 

In Section \ref{sec:badseed:proof} we prove the main result of Section~\ref{section:renorm} (see Theorem~\ref{thm:badseed:proba}) verifying the conditions under which 
high scale regions are favorable with very high probability. 
We use the weak decorrelation inequalities to prove that a good upper bound on the probability of unfavorable events on a certain scale results in an even
better bound on the next scale. 

In Section \ref{section:shape_proof} we prove Theorem \ref{thm:shapeThm} using Theorem \ref{thm:Sinfty:chemdist} and
Kingman's subadditive ergodic theorem \cite{kingman} in a standard fashion.

\bigskip

Throughout the paper, $e_i$, $1\leq i\leq d$, denote the canonical basis in $\mathbb R^d$. 
Constants are denoted by $c$ and $C$. Their values may change from place to place. 
We omit the dependence of constants on $a$, $b$, and $d$, but reflect the dependence on other parameters in the notation.

\section{Applications}\label{sec:examples}

In this section we give a number of examples for which our conditions  and results hold. 
In particular, we show that Theorems~\ref{thm:Sinfty:chemdist} and \ref{thm:shapeThm} hold for 
\begin{itemize}\itemsep0pt
 \item[(a)]
the {\it random interlacements} at any level $u>0$ (see Section~\ref{sec:ri}), 
giving alternative, model independent proofs to some of the results in \cite{CP},
\item[(b)]
the {\it vacant set of random interlacements} at level $u$ in the (non-empty) regime of so-called ``local uniqueness'', 
which is believed to coincide with the whole supercritical phase (see Section~\ref{sec:vsri}), 
\item[(c)]
the {\it level sets of the Gaussian free field}, also in the (non-empty) regime of local uniqueness (see Section~\ref{sec:gff}).
\end{itemize}
The results in Sections~\ref{sec:vsri} and \ref{sec:gff} are new, and in particular, they cannot be derived
using the techniques of \cite{CP} (see Remark \ref{remark_cerny_popov} for further discussion).
Let us also point out that the above models (a), (b) and (c) give rise to random subsets of $\Z^d$ with \emph{polynomial decay of correlations}, but still they satisfy the strong concentration bound \eqref{eq:Sinfty:chemdist}.

In Section~\ref{sec:rwtorus} we apply the results of Section~\ref{sec:vsri} to study the (chemical) diameter of the largest connected component 
in the complement of the trace of random walk on a discrete torus.

\subsection{Bernoulli percolation}\label{sec:bp}

Bernoulli site percolation with parameter $u\in(p_c,1)$ on $\Z^d$, $d\geq 2$, satisfies all the conditions
 \p{} -- \ppp{} and \s{} -- \sss{} (see, e.g., \cite{Grimmett}).
In this case $\mathbb P^u$ is just the product measure on $(\{0,1\}^{\Z^d},\mathcal F)$ with 
$\mathbb P^u[\Psi_x = 1] = 1 - \mathbb P^u[\Psi_x = 0] = u$. 

The analogue of Theorem \ref{thm:Sinfty:chemdist} for the chemical distance on the infinite cluster
 of Bernoulli percolation in the whole supercritical phase was proved in \cite[Theorem 1.1]{antal_pisztora}, where
 the authors obtained an exponentially decaying upper bound.

\subsection{Random interlacements}\label{sec:ri}

Random interlacements $\mathcal I^u$ at level $u>0$ on $\Z^d$, $d\geq 3$,  
is a random subset of $\Z^d$,  
which arises as the local limit as $N \to \infty$ of the set of sites visited by a simple random walk on
the discrete torus $(\Z/N\Z)^d$, $d \geq 3$, when it runs up to time $\lfloor u  N^d \rfloor$, see \cite{SznitmanAM, windisch_torus}.  
The distribution of $\mathcal I^u$ is determined by the equations
\begin{equation}\label{def_eq:Iu_capa}
\mathbb P[\mathcal I^u\cap K = \emptyset] = e^{-u\cdot \mathrm{cap}(K)}, \quad\text{for any finite $K\subseteq\Z^d$,}
\end{equation}
where $\mathrm{cap}(K)$ denotes the discrete capacity of $K$. 
For any $u>0$, $\mathcal I^u$ is an infinite almost surely connected random subset of $\Z^d$ (see \cite[(2.21)]{SznitmanAM}) with polynomially decaying correlations 
\begin{equation}\label{eq:ri:correlations}
\left|\mathbb P[x,y\in\mathcal I^u] - \mathbb P[x\in\mathcal I^u]\cdot \mathbb P[y\in\mathcal I^u]\right|\asymp (1+|x-y|)^{2-d}~,\qquad x,y\in\Z^d 
\end{equation}
(see \cite[(1.68)]{SznitmanAM}).
Geometric properties of random interlacements have been extensively studied in the last few years 
(see \cite{CP,LT,PrShel,PT,RS,RS:Transience,RS:Disordered}). 

\begin{remark}\label{remark_cerny_popov}
The fact that the conclusions of Theorems \ref{thm:Sinfty:chemdist} and \ref{thm:shapeThm} hold for the 
random interlacements $\mathcal I^u$ at level $u>0$ has been recently established in \cite[Theorems 1.3 and 1.1]{CP}.
The key idea in the proofs of \cite{CP} is a certain refinement of the strategy developed in \cite{RS}, 
which crucially relies on the fact that for any finite $K\subseteq\Z^d$ the random set $\mathcal{I}^u \cap K$ 
 can be generated as the union of independent random walk traces on $K$,
see \cite[(1.53)]{SznitmanAM}. 
The goal of this section is to observe (based on earlier results about random interlacements) that 
the laws of $\mathcal{I}^u$ satisfy conditions \p{} -- \ppp{} and \s{} -- \sss{}. 
As a result, we obtain an alternative, model independent approach to the results of \cite{CP}.
\end{remark}

\begin{theorem}\label{thm:ri}
For any $a>0$, the distributions of $\mathcal I^u$, $u\in(a,\infty)$, satisfy conditions \p{} -- \ppp{} and \s{} -- \sss{}. 
In particular, the results of Theorems~\ref{thm:Sinfty:chemdist} and \ref{thm:shapeThm} hold for
$\mathbb P^u$ being the distribution of $\mathcal I^u$ for any $u>0$. 
\end{theorem}
\begin{remark}
Theorem \ref{thm:Sinfty:chemdist} applied to the distribution of
 $\mathcal I^u$ gives a weaker result than \cite[Theorem 1.3]{CP}, 
since the bound on the right-hand side of \eqref{eq:Sinfty:chemdist} is not as good as the stretched exponential bound in \cite[Theorem 1.3]{CP}. 
Nevertheless, this weaker bound is still sufficient for us to deduce Theorem~\ref{thm:shapeThm}, 
thus giving an alternative proof of \cite[Theorem 1.1]{CP}.
\end{remark}
\begin{proof}[Proof of Theorem~\ref{thm:ri}]
The distributions of $\mathcal I^u$, $u>0$, satisfy condition \p{} by \cite[(2.3)]{SznitmanAM}, and condition \pp{} by \cite[(1.53)]{SznitmanAM}. 

The fact that the distributions of $\mathcal I^u$, $u\in(a,\infty)$, satisfy condition \ppp{} for any given $a>0$ 
follows from the more general decoupling inequalities \cite[Theorem~2.1]{Sznitman:Decoupling} applied to the graph $E = \Z^{d-1}\times\Z$ (see also \cite[Remark~2.7(1)]{Sznitman:Decoupling}). 
We will use the notation introduced in \cite[Section 2]{Sznitman:Decoupling}.
We take $\constP = 1/4$ in \eqref{eq:uwidehatu}. 
Our choice of parameters in \cite[Theorem~2.1]{Sznitman:Decoupling} is the following: 
\[
K = 1,\quad n= 0,\quad L_0 = L,\quad \ell_0 = \lfloor|x_1 - x_2|/L\rfloor,\quad \nu = d-2~(\alpha = d-1,~\beta=2), \quad \nu' = \nu/4 .\
\]
We also use the notation $\widehat u$ instead of $u'$ here. 
Note that with the above choice of parameters we have $\nu\geq 1$ and $\ell_0\geq R$ (since $d\geq 3$ and $|x_1 - x_2| \geq R\cdot L$). 
In \cite[Theorem~2.1]{Sznitman:Decoupling}, the inequality $l_0 \geq c(K, \nu')$ is required, so we need to choose
$\RP \geq c(1, \frac{d-2}{4})$.
With the above choice of parameters, $u$ and $\widehat u$ satisfy condition (2.7) of \cite[Theorem~2.1]{Sznitman:Decoupling} for all $R\geq \RP$ (for some suitably big $\RP$), 
so \eqref{eq:decorrelation:decreasing} and \eqref{eq:decorrelation:increasing} 
immediately follow from \cite[(2.8) and (2.9)]{Sznitman:Decoupling}
with the function $\funcP(L) = 2aL - \ln 2$,  thus the distributions of $\mathcal I^u$, $u\in(a,\infty)$  
indeed satisfy condition \ppp{}.

\medskip

By \cite[(2.21)]{SznitmanAM}, for any $u>0$, $\mathcal I^u$ is an almost surely connected infinite subset of $\Z^d$,
in particular we have $\mathcal{I}^u_r=\mathcal{I}^u$ (see \eqref{def:setr}) for any $r \in [0, \infty]$. 
The distributions of $\mathcal I^u$, $u>0$, satisfy \s{} by \cite[Proposition~1]{RS:Transience}, and \sss{} by \eqref{def_eq:Iu_capa} 
(since $\mathbb P[0\in\mathcal I^u] = 1 - e^{-u\mathrm{cap}(\{0\})}$ is a positive and continuous function of $u\in(0,\infty)$). 

\medskip

We have checked that the distributions of $\mathcal I^u$ satisfy \p{} -- \pp{} and \s{} -- \sss{} for all $u>0$, and \ppp{} for $u>a>0$ (and any given $a>0$). 
The proof of Theorem~\ref{thm:ri} is complete. 
\begin{remark}
Property \ppp{} can also be derived using the recent result of \cite[Theorem 1.1]{PoTe}. 
\end{remark}
\end{proof}

\subsection{Vacant set of random interlacements}\label{sec:vsri}

The vacant set of random interlacements $\mathcal V^u$ at level $u$ in dimension $d\geq 3$ is defined as the complement of random interlacements $\mathcal I^u$ at level $u$:
\begin{equation}\label{def:vsri}
\mathcal V^u = \Z^d \setminus \mathcal I^u,\quad u>0 .\
\end{equation}
In particular, it follows from \eqref{def_eq:Iu_capa} that the distribution of $\mathcal V^u$ is determined by the equations
\[
\mathbb P[\mathcal V^u\supset K] = e^{-u\cdot \mathrm{cap}(K)}, \quad\text{for any finite $K\subseteq\Z^d$,}
\]
and it follows from \eqref{eq:ri:correlations} that the correlations in $\mathcal V^u$ decay polynomially. 

In the same way as random interlacements, 
the vacant set at level $u>0$ arises as the local limit as $N \to \infty$ of the set of sites not visited by a simple random walk (we call it the vacant set, too) on
the discrete torus $(\Z/N\Z)^d$, $d \geq 3$, when it runs up to time $\lfloor u  N^d \rfloor$. 
In fact, the connection between the vacant sets of random interlacements and a simple random walk on $(\Z/N\Z)^d$ is 
in terms of a strong coupling of \cite[Theorem~1.1]{teixeira_windisch}. 
This coupling provided a powerful tool to study the fragmentation of $(\Z/N\Z)^d$ by a simple random walk using existing results about the geometry of $\mathcal V^u$. 
In particular, it was shown in \cite{teixeira_windisch} that the fragmentation question is intimately related to the 
existence of a non-trivial phase transition in $u$ for $\mathcal V^u$:
there exists $u_*\in(0,\infty)$ such that 
\begin{itemize}\itemsep0pt
\item[(i)] 
for any $u>u_*$, almost surely, all connected components of $\mathcal V^u$  are finite, and 
\item[(ii)]
for any $u<u_*$, almost surely, $\mathcal V^u$ contains an infinite connected component. 
\end{itemize}
The fact that $u_*<\infty$ was proved in \cite{SznitmanAM}, and the positivity of $u_*$ was established in \cite{SznitmanAM} when $d\geq 7$, 
and later in \cite{SidoraviciusSznitman_RI} for all $d\geq 3$. 
It is also known (see \cite{Teixeira_AAP}) that if there exists an infinite connected component in $\mathcal V^u$, then it is almost surely unique.

\medskip

In this section we focus on the supercritical phase of $\mathcal V^u$ (regime (ii) above). More specifically, 
we would like to derive conclusions of Theorems~\ref{thm:Sinfty:chemdist} and \ref{thm:shapeThm} for the vacant set
 $\mathcal V^u$. 
While it would be desirable to obtain results for all $u<u_*$, our understanding of the supercritical
phase is not yet good enough to do so. 
In Theorem~\ref{thm:vsri} we prove that there exists $\overline u\in(0,u_*]$ such that 
the conclusions of Theorems~\ref{thm:Sinfty:chemdist} and \ref{thm:shapeThm} hold for $\mathcal V^u$ with $u\in(0,\overline u)$.
We believe that $\overline u = u_*$ (see Remark~\ref{rem:ubarustar}). 
Even in the present form, the result of Theorem~\ref{thm:vsri} below is new for any choice of dimension $d\geq 3$ and level $u>0$. 

\begin{theorem}\label{thm:vsri}
There exists $\overline u\in(0,u_*]$ such that 
for any $1/\overline u<b<\infty$, the distributions $\mathbb P^v$ of $\mathcal V^{1/v}$, $v\in(1/\overline u,b)$, 
satisfy conditions \p{} -- \ppp{} and \s{} -- \sss{}. 
In particular, the conclusions of Theorems~\ref{thm:Sinfty:chemdist} and \ref{thm:shapeThm} hold 
for the law of  $\mathcal V^u$ for any $u\in(0,\overline u)$. 
\end{theorem}
\begin{proof}[Proof of Theorem~\ref{thm:vsri}]
For $v\in(0,\infty)$, let $\mathbb P^v$ denote the distribution of $\mathcal V^{1/v}$. 
Recall from \eqref{def:vsri} that for any $v>0$, $\mathcal V^{1/v} = \Z^d\setminus \mathcal I^{1/v}$. 
From this and Theorem~\ref{thm:ri} it is immediate that the family $\mathbb P^v$, $v\in(0,\infty)$, satisfies conditions \p{} and \pp{}, and 
the family $\mathbb P^v$, $v\in(0,b)$, satisfies condition \ppp{} for any $b<\infty$. 

\medskip

Let $\overline u$ be the largest $u$ such that the family $\mathbb P^v$, $v\in(1/u,\infty)$, satisfies condition \s{}. 
It follows from \eqref{eq:Sinfty} that $\overline u\leq u_*$. 
The positivity of $\overline u$ for all $d\geq 3$ follows from the main result of \cite{DRS} 
(for $d\geq 5$, it also follows from \cite[(1.2) and (1.3)]{Teixeira}). 

Finally, the family $\mathbb P^v$, $v\in(1/u_*,\infty)$, satisfies condition \sss{} by \cite[Corollary~1.2]{Teixeira_AAP}. 

\medskip

We have checked that the distributions of $\mathcal V^{1/v}$ satisfy \p{} -- \pp{} for $v\in(0,\infty)$, 
\ppp{} for $v\in(0,b)$ (and any given $b<\infty$), \s{} for $v\in(1/\overline u,\infty)$ (with $\overline u>0$), and \sss{} for $v\in(1/u_*,\infty)$.
Since $\overline u\leq u_*$, the proof of Theorem~\ref{thm:vsri} is complete. 
\end{proof}

\begin{remark}\label{rem:ubarustar}
Assumption \s{} is satisfied by  Bernoulli percolation in the whole supercritical phase 
(see e.g. \cite[(7.89),(8.98)]{Grimmett}). Thus it is reasonable to conjecture
that $\overline u = u_*$, i.e., that the distributions of $\mathcal V^u$ satisfy \s{} for all $u<u_*$. 
The verification of this conjecture would in particular imply that 
the conclusions of Theorems~\ref{thm:Sinfty:chemdist} and \ref{thm:shapeThm} hold 
for $\mathcal V^u$ for any $u\in(0,u_*)$.
\end{remark}

\subsection{Level sets of the Gaussian free field}\label{sec:gff}

The Gaussian free field on $\Z^d$, $d\geq 3$, is a centered Gaussian field $\varphi = (\varphi_x)_{x\in\Z^d}$
under the probability measure $\mathbb P$ with covariance 
$\mathbb E [\varphi_x\varphi_y] = g(x,y)$, for all $x,y\in\Z^d$, where $g(\cdot,\cdot)$ denotes the Green function of the simple random walk on $\Z^d$.  The random field $\varphi$ exhibits long-range correlations, since $g(0,x)$ decays like 
$|x|^{2-d}$ as $|x| \to \infty$.

For any $h\in\R$, we define the {\it excursion set} above level $h$ as 
\begin{equation}\label{def:levelsGFF}
E^{\geq h}_\varphi = \{x\in\Z^d~:~\varphi_x \geq h\} .\
\end{equation}
We view $E^{\geq h}_\varphi$ as a random subgraph of $\Z^d$. 
For any $d\geq 3$, there exists $h_*\in[0,\infty)$ such that 
\begin{itemize}\itemsep0pt
 \item[(a)]
for any $h< h_*$, $\mathbb P$-almost surely $E^{\geq h}_\varphi$ contains a unique infinite connected component, and 
\item[(b)]
for any $h>h_*$, $\mathbb P$-almost surely all the connected components of $E^{\geq h}_\varphi$ are finite. 
\end{itemize}
The finiteness of $h_*$ was established in \cite{BLM} when $d= 3$, 
and later in \cite{RodSz} for all $d\geq 3$. 
The non-negativity of $h_*$ was shown in \cite{BLM}, and the uniqueness in \cite[Remark 1.6]{RodSz}. 

\medskip

In this section we show that  the excursion set 
$E^{\geq h}_\varphi$
 of the Gaussian free field satisfies the assumptions and results of this paper in  a
  certain sub-regime $(-\infty, \overline{h})$ of the supercritical phase $(-\infty,h_*)$.
The relation between $\overline h$ and $h_*$ is discussed in Remark  \ref{h_bar_h_star}.

The main result of this section is the following theorem. 
\begin{theorem}\label{thm:gff}
For $u>0$, let $h(u) = h_*-u+1$, and let $\mathbb P^u$ be the law of $E^{\geq h(u)}_\varphi$. 
Then for any $d\geq 3$, there exists $\overline u\in [1,\infty)$ such that 
the family of distributions $\mathbb P^u$, $u>\overline u$, satisfies assumptions \p{} -- \ppp{} and \s{} -- \sss{}. 
In particular, there exists $\overline{h} \in (-\infty, h_*]$ such that
the conclusions of Theorems~\ref{thm:Sinfty:chemdist} and \ref{thm:shapeThm} hold for the distribution of
the excursion set $E^{\geq h}_\varphi$ for any $h \in (-\infty, \overline{h})$. 
\end{theorem}
\begin{proof}
The assumption \p{} is satisfied by $\mathbb P^u$, $u>0$, see the discussion above \cite[Lemma 1.5]{RodSz}.
 
The family $\mathbb P^u$, $u>0$, satisfies \pp{}, since 
the inclusion $E^{\geq h}_\varphi \subseteq E^{\geq h'}_\varphi$ holds for all $h' \leq h \in \R$ by \eqref{def:levelsGFF}.
 
The fact that the distributions $\mathbb P^u$, $u\in(1,\infty)$, satisfy \ppp{} follows from Lemma~\ref{lemma:decorr_gff} below.

\medskip

Recall that $\mathcal F$ denotes the canonical sigma-algebra on $\{0,1\}^{\Z^d}$, 
and denote by $\mathcal{G}$ the canonical sigma-algebra on $\R^{\Z^d}$. 
For an event $A$ in $\mathcal{F}$ or $\mathcal{G}$, we denote by $\mathds{1}(A)$ the indicator of $A$.
For any $B \in \mathcal F$ and $h \in \R$, we define the event $B^h \in \mathcal G$ by
\begin{equation}\label{def_eq_B_h}
 \forall \; \varphi \in \R^{\Z^d} \, : \qquad 
 \varphi \in  B^h  \; \; \iff \; \;  
  \left( \mathds{1} \left( \varphi_x \geq h \right),~x \in \Z^d\right) \in B.
  \end{equation}

\begin{lemma}\label{lemma:decorr_gff}

Let $d \geq 3$. 
Let $L \geq 1$ and $R\geq 50$ be integers. Let $x_1,x_2\in\Z^d$. 
For $i\in\{1,2\}$, let $A_i\in\sigma(\Psi_y~:~y\in \ballZ(x_i,10L))$ be decreasing events, and 
$B_i\in\sigma(\Psi_y~:~y\in \ballZ(x_i,10L))$ increasing events. 
Let $h,\widehat h\in\R$. 
There exist $c>0$ and $C<\infty$ such that if
\begin{equation*}
\widehat h \geq h + C \cdot R^{2-d}\qquad\mbox{and}\qquad  |x_1 - x_2|_\infty \geq R\cdot L ,\
\end{equation*}
then
\begin{equation}\label{eq:decorrelation:decreasing_gff}
\mathbb P\left[A_1^h\cap A_2^h\right] \leq 
\mathbb P\left[A_1^{\widehat h}\right] \cdot
\mathbb P\left[A_2^{\widehat h}\right] 
+ Ce^{-L^c} 
\end{equation}
and
\begin{equation}\label{eq:decorrelation:increasing_gff}
\mathbb P\left[B_1^{\widehat h}\cap B_2^{\widehat h}\right] \leq 
\mathbb P\left[B_1^h\right] \cdot
\mathbb P^u\left[B_2^h\right] 
+ Ce^{-L^c} .\
\end{equation}
\end{lemma}
The proof of Lemma~\ref{lemma:decorr_gff} is similar to that of \cite[Proposition 2.2]{RodSz}, 
but some new ideas are needed 
here since
the latter is not
strong enough in order to imply \ppp{} --- see also the issue discussed in \cite[Remark 2.3 (3)]{RodSz}. 
We postpone the proof of Lemma \ref{lemma:decorr_gff} to after finishing the rest of the proof of Theorem \ref{thm:gff}.
Note that 
one could use the more recent decoupling inequality of \cite{PR} to obtain \ppp{} --- however, 
for self-containedness we
include Lemma \ref{lemma:decorr_gff} and its short proof.

\medskip

To see that Lemma~\ref{lemma:decorr_gff} implies that \ppp{} is satisfied by the family $\mathbb P^u$, $u\in(1,\infty)$,  
note that for any $u,\widehat u>1$ such that $u\geq (1+C\cdot R^{2-d})\widehat u$, we have that $u\geq \widehat u + C\cdot R^{2-d}$. 
In particular, we get that $h(\widehat u) = h_* - \widehat u + 1 \geq h(u) + C\cdot R^{2-d}$. 
It is now immediate from Lemma \ref{lemma:decorr_gff} that $\mathbb P^u$, $u \in (1, \infty),$ satisfies \ppp{} for
 any choice of $0<\epsP<1$ and $0<\constP<d-2$.

\medskip

We proceed with \s{}. 
Let $\overline u\in[0,\infty]$ be the smallest $u'$ such that the family $\mathbb P^u$, $u>u'$, satisfies condition \s{}. 
Define $\overline{h}=h(\overline{u})$ and note that $\overline{h} \leq h_*$ (i.e., $\overline{u} \geq 1$) follows from
\eqref{eq:Sinfty} and the definition of $h_*$.
We will now prove that $\overline u < \infty$ (i.e., $\overline{h}> - \infty$).
 
We say that $x,y \in \Z^d$ are $*$-connected in $\mathcal S$ if there exists
$x_1,\dots,x_n \in \mathcal S$ such that $x_1=x$, $x_n=y$ and $|x_{k+1}-x_k|_{\infty}=1$ for all 
 $ k \in \{1,\dots,  n-1\}$.
 Essentially the same proof as the one of \cite[(2.64)]{RodSz} implies that for any $d \geq 3$ there exists 
 $\widetilde{h} \in (0,\infty)$ and constants $c>0$ and $C<\infty$ such that 
 \begin{equation*}
 \mathbb{P} \left[ \text{ $0$ and $x$ are $*$-connected in $E^{\geq h}_\varphi$ } \right] \leq C e^{-|x|^c}, 
 \quad \text{ for all $x \in \Z^d$, $h \geq \widetilde{h}$.}
 \end{equation*}

Observing that the laws of $E^{\geq h}_\varphi$ and $\Z^d\setminus E^{\geq -h}_\varphi$ are the same under $\mathbb P$, 
we immediately obtain that the $*$-connected components in the complement of $E^{\geq -h}_\varphi$ are very small
for any $h \in (\widetilde{h},\infty)$. 
Then, a standard argument based on the connectedness of $*$-boundaries (see, e.g., the proof of \cite[Corollary~3.7]{DRS}) 
implies that the distribution of $E^{\geq -h}_\varphi$ satisfies \s{} for any $h \in (\widetilde{h},\infty)$. 
Let $\widetilde u = h_* + \widetilde h + 1< \infty$. 
For any $u>\widetilde u$, we have that $h(u) < - \widetilde h$. 
Therefore, the family $\mathbb P^u$, $u>\widetilde u$ satisfies condition \s{}. In particular, $\overline u \leq \widetilde u < \infty$. 

\medskip

Finally, recall that by \cite[Remark~1.6]{RodSz}, for any $h<h_*$, $\mathbb P$-almost surely, there is a unique infinite connected component in $E^{\geq h}_\varphi$. 
Therefore, with standard methods (see, e.g., the proof of \cite[Lemma (8.10)]{Grimmett}), it follows that 
the distributions of $\mathbb P^u$, $u \in (1, \infty)$, satisfy assumption \sss{}.

We have checked that the family of distributions $\mathbb P^u$
 satisfies \p{} -- \pp{} for $u>0$, 
\ppp{} and \sss{} for $u\in(1,\infty)$, and \s{} for $u\in (\overline u,\infty)$ with $\overline u<\infty$.
The proof of Theorem~\ref{thm:gff} is complete, given the result of Lemma \ref{lemma:decorr_gff}.
\end{proof}

\begin{remark}\label{h_bar_h_star}
Similarly to Remark \ref{rem:ubarustar},
 it is reasonable to conjecture
that $\overline h = h_*$, i.e., that the distribution of $E^{\geq h}_\varphi$ satisfies \s{} for all $h<h_*$. 
As soon as this conjecture  is verified, 
the conclusions of Theorems~\ref{thm:Sinfty:chemdist} and \ref{thm:shapeThm}
extend to the excursion set  $E^{\geq h}_\varphi$ for any parameter 
$h$ in the supercritical regime $(-\infty, h_*)$.
Proving that $\overline h = h_*$ may be quite hard. 
A reasonable start would be to prove that $\overline h\geq 0$. 
A strict inequality, i.e., $\overline h>0$, would be even better, but at the moment there is no rigorous argument even for the statement that $h_* >0$ for all $d\geq 3$. 
It is only known that $h_* >0$ when the dimension $d$ is large, see \cite[Theorem 3.3]{RodSz}. 
Extending the result of \cite[Theorem 3.3]{RodSz} to all dimensions $d\geq 3$ is 
thus another interesting open problem, see \cite[Remark 3.6 (3)]{RodSz}. 
\end{remark}

\begin{proof}[Proof of Lemma \ref{lemma:decorr_gff}]
We will only prove \eqref{eq:decorrelation:increasing_gff}. The proof of \eqref{eq:decorrelation:decreasing_gff} is similar, and we omit it. 

Let $K_i=\ballZ(x_i,10L)$ for $i=1,2$.
Let $H(h)$ be the event that 
\begin{equation*}
H(h) = \left\{ \text{ there is no nearest neighbour path in $E^{\geq h}_\varphi$ 
connecting $K_1$ to $\ballZ(x_1,20L)^c$} \right\} .\
\end{equation*}
By \cite[(2.63)]{RodSz}, there exists $h_0\in(h_*,\infty)$ and constants $c>0$ and $C<\infty$ such that 
$\mathbb{P} [H(h_0)^c] \leq  Ce^{-L^c}$ for all $L\geq 1$. 
From now on we fix such $h_0$ and write $H$ for $H(h_0)$. 
In order to prove \eqref{eq:decorrelation:increasing_gff} we only need to show that there exists a constant $C<\infty$ such that
with $\gamma$ defined as 
\begin{equation}\label{def_eq_lambda}
\gamma = C \cdot R^{2-d}, 
\end{equation}
for every $h \in \R$ and every pair $B_i\in\sigma(\Psi_y~:~y\in K_i)$, $i=1,2$ of increasing events in $\mathcal F$ satisfying the assumptions of Lemma \ref{lemma:decorr_gff}, 
we have 
\begin{equation}\label{eq:decorrelation:increasing_gff_2}
\mathbb P \left[B_1^h \cap B_2^h\cap H \right] \leq 
\mathbb P \left[ B_1^{h-\gamma} \right] \cdot
\mathbb P \left[ B_2^{h-\gamma} \right] .\
\end{equation}
(Here, to simplify notation, 
we replaced $\widehat h$ of \eqref{eq:decorrelation:increasing_gff} by $h$, and $h$ of \eqref{eq:decorrelation:increasing_gff} by $h-\gamma$, hopefully without confusion.)

Denote by $\mathcal{D}$ the union of $K_1$ and the set of vertices of $\Z^d$ connected to $K_1$ by a simple path in $E^{\geq h_0}_\varphi$. 
Note that if $H$ occurs, then 
\begin{equation*}
\ballZ(x_1,10L) \subseteq \mathcal{D} \subseteq \ballZ(x_1,20L). 
\end{equation*}
Denote by $\mathbf{D}$ the set of subsets $D$ of $\ballZ(x_1,20L)$ that can arise as a realization of $\mathcal{D}$ with positive $\mathbb{P}$-probability. 
We start proving \eqref{eq:decorrelation:increasing_gff_2} by writing
\begin{equation}\label{chopping}
\mathbb P \left[B_1^h \cap B_2^h \cap H\right] =
 \sum_{D \in \mathbf{D}} \mathbb{P} \left[ B_1^h \cap B_2^h  \cap \left\{ \mathcal{D}=D  \right\} \right] .\
\end{equation}
Note that the event $\left\{ \mathcal{D}=D  \right\}$ is measurable with respect to the
 sub-sigma-algebra of $\mathcal G$ generated by 
$\varphi_{\overline{D}}= \left( \varphi_{x} \right)_{x \in \overline{D}}$, where 
$\overline{D}=\{ x \in \Z^d \, : \, \exists \, y \in D \; : \; |x-y|_1 \leq 1 \} $ denotes the $l_1$-closure of $D$. 
Thus for any $D \in \mathbf{D}$, 
\begin{equation}\label{expect_of_condprob}
\mathbb{P} \left[ B_1^h \cap B_2^h  \cap \left\{ \mathcal{D}=D  \right\} \right]=
\mathbb{E} \left[ \mathds{1}(B_1^h) \cdot \mathds{1}\left( \mathcal{D}=D\right) \cdot
\mathbb{P} \left[ \left( \varphi_x \right)_{x \in K_2} \in  B_2^h \, \Big\vert \, \varphi_{\overline{D}} \; \right] \,  \right].
\end{equation}
According to \cite[Remark 1.3]{RodSz}, for every $D \in \mathbf{D}$ there exists a probability measure  
$\widetilde{\mathbb{P}}$ on $\R^{\Z^d}$ such that $\left( \widetilde{\varphi}_x \right)_{x \in \Z^d}$ is a
 centered Gaussian field under $\widetilde{\mathbb{P}}$ with $\widetilde \varphi_x = 0$ for all $x\in \overline D$, and 
 \begin{equation}\label{conditioning_of_gff}
 \mathbb{P} \left[  \left( \varphi_x \right)_{x \in \Z^d} \in \, \cdot \; \big\vert \, \varphi_{\overline{D}} \;  \right]=
 \widetilde{\mathbb{P}} \left[  \left( \widetilde{\varphi}_x + \mu_x \right)_{x \in \Z^d} \in \, \cdot \;   \right],
 \end{equation}
where $\mu_x$, $x \in \Z^d$ is given by
\begin{equation*}
 \mu_x = \sum_{y \in \overline{D}} 
P_x \left[ H_{\overline{D}} < \infty, \,  X_{H_{\overline{D}}} =y   \right] \cdot \varphi_y, 
\end{equation*}
where $P_x$ denotes the law of simple random walk on $\Z^d$ started at $x$,
$H_{\overline{D}}$ denotes the first time that the random walk enters $\overline{D}$ and
 $X_{H_{\overline{D}}}$ is the position of the walker at time $H_{\overline{D}}$. 

Let us fix $D \in \mathbf{D}$ and $\varphi_{\overline{D}}$ such that the event $\left\{ \mathcal{D}=D  \right\}$ occurs.
By the definition of $\mathcal{D}$, we have that $\mu_x < h_0\cdot P_x(H_{\overline D}<\infty)$ for all $x\notin \overline D$. Under the assumptions of Lemma \ref{lemma:decorr_gff}
the $l^\infty$-distance of $K_2$ and $\overline{D}$ is at least $ \frac15 R \cdot L$, thus
by a standard argument based on estimates for the discrete Green function and capacity (see, e.g., the calculation below \cite[(2.30)]{RodSz}), 
we obtain that
 there exists $C'<\infty$ such that for any $x\in K_2$, 
$P_x(H_{\overline D}<\infty) \leq C'\cdot R^{2-d}$. In particular, by taking $C$ in \eqref{def_eq_lambda} to be
 $2  h_0 \cdot C'$, we obtain that 
 \begin{equation}\label{upperbound_on_mu_x_gamma}
\max_{x \in K_2} \mu_x \leq h_0\cdot C' \cdot R^{2-d} \stackrel{\eqref{def_eq_lambda}}{=} \frac12 \gamma .\
 \end{equation}

Now we are ready to carry out the ``sprinkling'' mentioned in Remark \ref{rem:decorrelation}:
\begin{multline}\label{eq_sprinkling}
\mathbb{P} \left[ \left( \varphi_x \right)_{x \in K_2} \in  B_2^h \, \Big\vert \, \varphi_{\overline{D}} \; \right]
\stackrel{\eqref{conditioning_of_gff}}{=}
\widetilde{\mathbb{P}} \left[  \left( \widetilde{\varphi}_x + \mu_x \right)_{x \in K_2} \in B_2^h \;   \right]
\stackrel{ (*) }{\leq}\\
\widetilde{\mathbb{P}} \left[  \left( \widetilde{\varphi}_x + \frac12 \gamma \right)_{x \in K_2} \in B_2^h \;   \right]=
\widetilde{\mathbb{P}} \left[  \left( \widetilde{\varphi}_x - \frac12 \gamma \right)_{x \in K_2} \in B_2^{h-\gamma} \;   \right] \stackrel{ (*) }{\leq}
\widetilde{\mathbb{P}} \left[  \left( \widetilde{\varphi}_x - \mu_x \right)_{x \in K_2} \in B_2^{h-\gamma} \;   \right],
\end{multline}
where in the equations marked by $(*)$ above we used \eqref{upperbound_on_mu_x_gamma} and the fact that $B_2$ is an increasing event. 

\medskip

For any $B \in \mathcal F$ we define the \emph{flipped} event $\check{B} \in \mathcal F$ by
\begin{equation}\label{def_eq_flipped_event}
\forall \; \atom \in \{0,1\}^{\Z^d} \, : \qquad 
\atom \in \check{B} \;\; \iff \;\; \left( \mathds{1} ( \atom_x =0 ), x\in\Z^d \right) \in B.    
\end{equation}
Note that if $B$ is an increasing event, then $\check{B}$ is a decreasing event.
With definition \eqref{def_eq_flipped_event} at hand, we note that  
\begin{multline}\label{flipped_identity}
\left\{ \left( \widetilde{\varphi}_x - \mu_x \right)_{x \in K_2} \in B_2^{h-\gamma} \right\}
  \stackrel{\eqref{def_eq_B_h}}{=} 
\left\{ \left( \mathds{1}\left( \widetilde{\varphi}_x - \mu_x \geq h -\gamma\right), x\in K_2   \right) \in B_2 \right\}
 \stackrel{\eqref{def_eq_flipped_event} }{=} \\
\left\{ \left( \mathds{1}( \widetilde{\varphi}_x - \mu_x < h -\gamma ), x\in K_2   \right) \in \check{B}_2
\right\}
\stackrel{\eqref{def_eq_B_h}}=
\left\{ \left( -\widetilde{\varphi}_x + \mu_x \right)_{x \in K_2} \in \check{B}_2^{\gamma-h} \right\}, 
\quad \widetilde{\mathbb{P}}\text{-a.s.}
\end{multline}
Plugging this  identity into \eqref{eq_sprinkling} we obtain
\begin{equation*}
\mathbb{P} \left[   B_2^h \, \Big\vert \, \varphi_{\overline{D}} \; \right]
\leq 
\widetilde{\mathbb{P}} \left[  \left( -\widetilde{\varphi}_x + \mu_x \right)_{x \in K_2} \in \check{B}_2^{\gamma-h} \;   \right] 
\stackrel{(*)}{=}
\widetilde{\mathbb{P}} \left[  \left( \widetilde{\varphi}_x + \mu_x \right)_{x \in K_2} \in \check{B}_2^{\gamma-h} \;   \right] 
\stackrel{\eqref{conditioning_of_gff}}{=}
\mathbb{P} \left[ \check{B}_2^{\gamma-h} \, \Big\vert \, \varphi_{\overline{D}} \; \right],
\end{equation*}
where in the equation marked by $(*)$ we used the fact that $\widetilde{\varphi}_{K_2}$ and $-\widetilde{\varphi}_{K_2}$
have the same distribution under $\widetilde{\mathbb{P}}$.
Substituting this inequality back into \eqref{expect_of_condprob}, we obtain 
\[ \mathbb{P} \left[ B_1^h \cap B_2^h  \cap \left\{ \mathcal{D}=D  \right\} \right] \leq 
\mathbb{P} \left[ B_1^h \cap  \check{B}_2^{\gamma-h}  \cap \left\{ \mathcal{D}=D  \right\} \right]. 
\]
Combining this inequality with \eqref{chopping} we conclude
\[
\mathbb P \left[B_1^h \cap B_2^h \cap H\right] 
\leq
\mathbb P \left[B_1^h \cap \check{B}_2^{\gamma-h}  \right]
 \stackrel{(a)}{\leq}
\mathbb P \left[B_1^h \right] \mathbb P \left[ \check{B}_2^{\gamma-h}  \right]
 \stackrel{(b)}{\leq} 
\mathbb P \left[B_1^{h-\gamma} \right] \mathbb P \left[ B_2^{h-\gamma}  \right],
\]
where in $(a)$ we used the fact that the increasing event $B_1^h$ and the decreasing event $\check{B}_2^{\gamma-h}$
are negatively correlated under $\mathbb P$  by the
FKG-inequality for the Gaussian free field (see the remark above 
\cite[Lemma 1.4]{RodSz}) and in $(b)$ we used the fact that $B_1$ is increasing and a similar identity as in \eqref{flipped_identity} combined with the
fact that  $\varphi$ and $-\varphi$ have the same distribution under $\mathbb{P}$. This concludes the proof of \eqref{eq:decorrelation:increasing_gff_2} and 
\eqref{eq:decorrelation:increasing_gff}. The proof of \eqref{eq:decorrelation:decreasing_gff} is analogous and we omit it.
The proof of Lemma \ref{lemma:decorr_gff} is complete.
\end{proof}

\subsection{Vacant set of random walk on a torus}\label{sec:rwtorus}

We consider a simple symmetric nearest neighbor random walk on the $d$-dimensional torus $\mathbb{T}_N=(\Z / N \Z)^d$, $d\geq 3$, 
started from a uniformly distributed vertex. 
Let $\mathcal{I}^{u,N}$ denote the random subset of $\mathbb{T}_N$ which consists of the sites visited by the random walk in the first $\lfloor u N^d \rfloor$ steps, 
and \[\mathcal{V}^{u,N}= \mathbb{T}_N \setminus \mathcal{I}^{u,N}\] the vacant set of the random walk. 
We denote the distribution of $\mathcal V^{u,N}$ by $\mathbb P^{u,N}$. 
We view $\mathcal{V}^{u,N}$ as a (random) graph by drawing an edge between any two vertices of $\mathcal{V}^{u,N}$ at $\ell^1$-distance $1$.
The study of percolative properties of $\mathcal{V}^{u,N}$ was initiated in \cite{BenjaminiSznitman} 
and recently significantly boosted in \cite{teixeira_windisch}.  Informally, \cite[Theorems 1.2 and 1.3]{teixeira_windisch} 
state that for all $d \geq 3$, with high probability as $N\to\infty$, 
\begin{enumerate}[(a)]\itemsep0pt
\item if $u$ is big enough then $\mathcal{V}^{u,N}$ consists of small connected components, but
\item if  $u$ is small enough then  $\mathcal{V}^{u,N}$ contains macroscopic connected components.
\end{enumerate}
 It is conjectured that the transition between these two phases is sharp and occurs at the critical 
threshold $u_*$ of the vacant set of random interlacements (see Section \ref{sec:vsri} for the definition of $u_*$).

Uniqueness of the giant  component of $\mathcal{V}^{u,N}$ is only known in high dimensions: it follows from 
 \cite[Theorem 1.4]{teixeira_windisch} that if $d \geq 5$ and $u$ is small enough  then 
 the connected component $\mathcal{C}^u_{max}$ of $\mathcal{V}^{u,N}$ with the largest volume has a positive asymptotic density and the volume of the second largest
 connected component  of $\mathcal{V}^{u,N}$ is bounded by $(\log N)^C$, with high probability as $N \to \infty$.

We are interested in the chemical distance $\chemdist{u,N}(\cdot ,\cdot)$ on the vacant set $\mathcal{V}^{u,N}$. 
In particular, we show in Theorem \ref{thm:chemdist:torus} below that if $d \geq 5$ and $u$ is small enough then 
the (chemical) diameter of the subgraph of $\mathbb{T}_N$ spanned by the giant vacant component $\mathcal{C}^u_{max}$ 
is comparable to $N$ with high probability as $N \to \infty$.

\begin{theorem}\label{thm:chemdist:torus}
For any $d\geq 5$, there exists $\widetilde u = \widetilde u(d)>0$ such that for each $u<\widetilde u$ 
there exists $C = C(d,u)<\infty$ such that for all $\kappa>0$,    
\begin{equation}\label{eq:chemdist:torus}
\lim_{N \to \infty} N^{\kappa} \left(
1 - \mathbb P^{u,N}\left[
\text{$N/C\leq\max_{x,y\in\mathcal C^u_{max}}\chemdist{u,N}(x,y) \leq C N$}
\right]\right)=0.\
\end{equation}
\end{theorem}

\begin{remark}
The analogous result about the chemical distances on $\mathcal{I}^{u,N}$
was proved in \cite[Theorem 1.6]{CP}, which improved the earlier result of \cite[Theorem 2.1]{PrShel}.
Similarly to Remark \ref{remark_cerny_popov}, we note that many of the model-specific methods used in the proofs of these
results about the connectivity of the random walk trace $\mathcal{I}^{u,N}$
cannot be applied for the complement $\mathcal{V}^{u,N}$ of the random walk trace.
\end{remark}

\begin{remark}
We believe that the statement of Theorem~\ref{thm:chemdist:torus} holds for 
any $d\geq 3$ and $u<u_*$ (see Remark \ref{rem:ubarustar}, as well as
Section~\ref{sec:vsri} for the definition of $u_*$). 
A possibly simpler open problem is to show that $\widetilde u(d)>0$ for $d\in\{3,4\}$. 
\end{remark}
The main ingredients for the proof of Theorem~\ref{thm:chemdist:torus} are Theorem \ref{thm:vsri},  
the strong coupling of \cite[Theorem 1.1]{teixeira_windisch} between the vacant set of random walk on the torus and the vacant set of random interlacements, 
and the notion of {\it strongly supercritical} values of $u$ in \cite[Definition 2.4]{teixeira_windisch}.

We split the proof of Theorem~\ref{thm:chemdist:torus} into several steps. 
We first show that $\mathcal V^{u,N}$ contains many mesoscopic (not necessarily connected) subsets $\mathcal C_z$, $z\in\mathbb T_N$, with good chemical distance properties.
We then show that these sets are connected to each other and to any vertex of $\mathcal V^{u,N}$ which is in a connected component of large enough diameter 
by paths of length at most $CN$. 
Most of the sets $\mathcal C_z$ will be at distance $>N/C$ from each other, which will guarantee the lower bound on the (chemical) diameter of $\mathcal C^u_{max}$.

\medskip

Recall the definition of the vacant set of random interlacements $\mathcal V^u$ from Section~\ref{sec:vsri}. 
We begin with a lemma, which is motivated by the strong coupling \cite[Theorem~1.1]{teixeira_windisch} of $\mathcal V^{u,N}$ and $\mathcal V^u$. 
The lemma states that under some assumptions on $u$, there exists a mesoscopic subset of $\mathcal V^{u,N}$ with good chemical distance properties.

Let us choose (somewhat arbitrarily) the mesoscopic scale
 \[n = N^{1/3}.\] 
\begin{lemma}\label{l:goodset}
Let $d\geq 3$, $u>0$, and $\varepsilon>0$. If the distribution of $\mathcal V^{(1+\varepsilon)u}$ satisfies \eqref{eq:C1:infty} and \eqref{eq:Sinfty:chemdist}, then 
there exists $C = C(u,\varepsilon)<\infty$ such that for any $\kappa>0$, 
\begin{equation}\label{eq:goodset}
\lim_{N\to\infty} N^\kappa \left(
1 - \mathbb P^{u,N} \left[
\begin{array}{c}
\text{there exists $\mathcal C\subseteq \mathcal V^{u,N}\cap\ballZ(0,n)$ such that}\\
\text{(i) for all $x,y\in\mathcal C$, $\chemdist{u,N}(x,y)\leq Cn$, and}\\ 
\text{(ii) for all $x\in\ballZ(0,n/2)$, $\ballZ(x,n^{1/d})\cap\mathcal C\neq \emptyset$}
\end{array}
\right]
\right)
=0 .\
\end{equation}
\end{lemma}
\begin{remark}\label{rem:goodset}
Note that for any $u<\overline u$ (see the statement of Theorem~\ref{thm:vsri}), $\mathcal V^u$ satisfies \eqref{eq:C1:infty} and \eqref{eq:Sinfty:chemdist}. 
In particular, we deduce from Lemma~\ref{l:goodset} and Theorem~\ref{thm:vsri} that for any $d\geq 3$, $u<\overline u$, and $\kappa>0$, \eqref{eq:goodset} holds 
(by taking $\varepsilon>0$ such that $(1+\varepsilon)u<\overline u$).
\end{remark}
\begin{proof}[Proof of Lemma~\ref{l:goodset}]
Let $A = \ballZ(0,n^2)\subset \mathbb T_N$ be the ball of radius $n^2 (= N^{2/3})$ centered at $0\in\mathbb T_N$. 
The ball $A$ is isomorphic to the ball $\mathrm A = \ballZ(0,n^2)\subset\Z^d$ via the graph isomorphism $\Phi_N:A\to\mathrm A$. 
It follows from \cite[Theorem 1.1]{teixeira_windisch} that for any $d\geq 3$, $u>0$, $\varepsilon>0$, $\kappa>0$, and $N\geq 1$, 
there exists a coupling
$(\hat{\Omega}, \hat{\mathcal{A}}, \hat{\mathbb{P}}^{u,N})$ of $\mathcal V^{u,N}$ and $\mathcal V^{(1+\varepsilon)u}$ 
such that 
\begin{equation}\label{eq:coupling}
\lim_{N\to\infty} N^\kappa \left(
1 - \hat{\mathbb{P}}^{u,N} \left[
\mathcal{V}^{(1+\varepsilon)u} \cap \mathrm{A} \subseteq \Phi_N(\mathcal{V}^{u,N} \cap A)
\right]
\right)
=0 .\
\end{equation}
Now we fix $u>0$ and $\varepsilon>0$ so that the distribution of $\mathcal V^{(1+\varepsilon)u}$ satisfies \eqref{eq:C1:infty} and \eqref{eq:Sinfty:chemdist}. 
Let $u' = (1+\varepsilon)u$, let $\chemdist{u'}(\cdot,\cdot)$ be the chemical distance in $\mathcal V^{u'}$, and 
let $C_{u'}<\infty$ be the constant from \eqref{eq:Sinfty:chemdist}. 
Let $\mathcal V^{u'}_\infty$ be the unique (by \eqref{eq:Sinfty:chemdist}) infinite connected component of $\mathcal V^{u'}$. 
Consider the event 
\[
\mathcal E^{u',N} = \left\{
\begin{array}{c}
\text{for any $x,y\in\mathcal V^{u'}_\infty\cap\ballZ(0,n)$, $\chemdist{u'}(x,y)\leq C_{u'}n$, and}\\
\text{for any $x\in\ballZ(0,n/2)$, $\ballZ(x,n^{1/d})\cap\mathcal V^{u'}_\infty\neq\emptyset$} 
\end{array}
\right\} .\
\]
It follows from \eqref{eq:C1:infty}, \eqref{eq:Sinfty:chemdist}, and our choice of $u'$ that for any $\kappa>0$, 
\begin{equation}\label{eq:EuN}
\lim_{N\to\infty} N^\kappa \left(1- \hat{\mathbb{P}}^{u,N}\left[\mathcal E^{u',N}\right]\right) = 0 .\
\end{equation}

Note that if the event 
$\mathcal E^{u',N}\cap\{\mathcal{V}^{u'} \cap \mathrm{A} \subseteq \Phi_N(\mathcal{V}^{u,N} \cap A)\}$
occurs, then the set 
\[
\mathcal C = \Phi_N^{-1}\left(\mathcal V^{u'}_\infty\cap\ballZ(0,n)\right) \subseteq \mathcal V^{u,N}\cap\ballZ(0,n)
\]
satisfies the conditions in \eqref{eq:goodset} for all large enough $N$. 
Indeed, it obviously satisfies (ii) by the definition of $\mathcal E^{u',N}$ and the fact that $\ballZ(x,n^{1/d})\subset\ballZ(0,n)$ for all $x\in\ballZ(0,n/2)$.
It satisfies (i) since for any $x,y\in\mathcal V^{u'}_\infty\cap\ballZ(0,n)$, 
by the definition of $\mathcal E^{u',N}$, there exists a path of length at most $C_{u'}n$ in $\mathcal V^{u'}$ between $x$ and $y$, 
which (if $n+C_{u'}n < n^2$) is contained in $\mathrm A\cap\mathcal V^{u'}$ for large enough $N$. 
Therefore, there must be a path  in $\mathcal V^{u,N}$ of length at most $C_{u'}n$ between any pair of vertices $x,y\in\mathcal C$ 
(just take the preimage of the corresponding path in $\mathcal V^{u'}_\infty$ under $\Phi_N$).

Lemma~\ref{l:goodset} now follows from \eqref{eq:coupling} and \eqref{eq:EuN}.
\end{proof}
\begin{corollary}\label{cor:goodset}
Let $d\geq 3$ and $0<u<\overline u$, where $\overline u$ is defined in Theorem~\ref{thm:vsri}.
 Choose $\varepsilon = \varepsilon (u)>0$ such that $(1+\varepsilon)u<\overline u$.
Let $C = C(u) = C(u,\varepsilon(u))<\infty$ be the constant from \eqref{eq:goodset}. 
Consider the event
\begin{equation}\label{def:FuN}
\mathcal F^{u,N} = 
\left\{
\begin{array}{c}
\text{for any $z\in\mathbb T_N$, there exists $\mathcal C_z\subseteq \mathcal V^{u,N}\cap\ballZ(z,n)$ such that}\\
\text{(i) for all $x,y\in\mathcal C_z$, $\chemdist{u,N}(x,y)\leq Cn$, and}\\ 
\text{(ii) for all $x\in\ballZ(z,n/2)$, $\ballZ(x,n^{1/d})\cap\mathcal C_z\neq \emptyset$}
\end{array}
\right\} .\
\end{equation}
It follows from Lemma~\ref{l:goodset} and Remark~\ref{rem:goodset} that for any $d\geq 3$, $u<\overline u$, and $\kappa>0$, 
\begin{equation}\label{eq:FuN}
\lim_{N\to\infty} N^\kappa \left(1- \mathbb{P}^{u,N}\left[\mathcal F^{u,N}\right]\right) = 0 .\
\end{equation}
\end{corollary}

Lemma~\ref{l:goodset} establishes under some assumptions on $u$ 
the existence of a subset $\mathcal C$ of $\mathcal V^{u,N}$ with good chemical distance properties. 
Those assumptions are not enough to extend the result of Lemma~\ref{l:goodset} to $\mathcal C^u_{max}$, 
and we need to assume that $\mathcal V^{u,N}$ is well connected locally. 
More precisely, let $\mathcal V^{u,N}_r$ be the subset of vertices in connected components of $\mathcal V^{u,N}$ with diameter $\geq r$. 
Consider the event
\[
\mathcal G^{u,N} = 
\left\{
\begin{array}{c}
\text{for all $x,y\in\mathcal V^{u,N}_{n^{1/d}}$ with $|x-y|\leq 2n^{1/d}$}\\
\text{$x$ and $y$ are connected in $\mathcal V^{u,N}\cap\ballZ(x,4n^{1/d})$}
\end{array}
\right\} .\
\]
The following lemma provides sufficient conditions for the connectedness and ubiquity of $\mathcal V^{u,N}_{n^{1/d}}$ 
and states that the (chemical) diameter of $\mathcal V^{u,N}_{n^{1/d}}$ is comparable to $N$. 
\begin{lemma}\label{l:inclusion}
For any $d\geq 3$ and $u>0$, there exists $C' = C'(u)<\infty$ and $N'<\infty$ such that for all $N\geq N'$, 
\begin{equation}\label{eq:inclusion}
\mathcal F^{u,N}\cap \mathcal G^{u,N} \subseteq 
\left\{
\begin{array}{c}
\text{for all $x,y\in\mathcal V^{u,N}_{n^{1/d}}$, $\chemdist{u,N}(x,y)\leq C'N$, and}\\
\text{for all $x\in\mathbb T_N$, $\mathcal V^{u,N}_{n^{1/d}}\cap\ballZ(x,n^{1/d}) \neq\emptyset$}
\end{array}
\right\} .\
\end{equation}
\end{lemma}
\begin{proof}[Proof of Lemma~\ref{l:inclusion}]
Fix $d\geq 3$ and $u>0$, and assume that the event $\mathcal F^{u,N}\cap\mathcal G^{u,N}$ occurs. 
First of all, note that for any $z\in\mathbb T_N$, by the definition \eqref{def:FuN} of $\mathcal C_z$, 
\begin{equation}\label{eq:CzVuN}
\mathcal C_z \subset \mathcal V^{u,N}_{n^{1/d}} .\
\end{equation}
Indeed, $\mathcal C_z$ is a subset of a connected set in $\mathcal V^{u,N}$ and contains vertices at distance $\geq n/2-2n^{1/d}>n^{1/d}$ 
(the last inequality holds if $N$ is large enough).
It follows from \eqref{eq:CzVuN} and the definitions of $\mathcal F^{u,N}$ and $\mathcal G^{u,N}$ that 
\begin{equation}\label{eq:yCx}
\text{for all $x\in\mathcal V^{u,N}_{n^{1/d}}$, there exists $y\in\mathcal C_x$ such that $\chemdist{u,N}(x,y) \leq |\ballZ(x,4n^{1/d})| \leq 9^d\cdot n$.}
\end{equation}
For $x,y\in\mathbb T_N$, let 
\[
\chemdist{u,N}(\mathcal C_x,\mathcal C_y) = \inf\{\chemdist{u,N}(x',y')~:~x'\in\mathcal C_x,~y'\in\mathcal C_y\} \in [0,\infty] .\
\]
Note that 
\begin{equation}\label{eq:CxCy:n}
\text{for any $x,y\in\mathbb T_N$ with $|x-y|<n$,~~ $\chemdist{u,N}(\mathcal C_x,\mathcal C_y)\leq 9^d\cdot n$.}
\end{equation}
Indeed, for any $x,y\in\mathbb T_N$ with $|x-y|<n$, there exists $z\in\ballZ(x,n/2)\cap\ballZ(y,n/2)$, therefore 
there exist $x'\in\mathcal C_x\cap\ballZ(z,n^{1/d})$ and $y'\in\mathcal C_y\cap\ballZ(z,n^{1/d})$. 
By \eqref{eq:CzVuN}, the definition of $\mathcal G^{u,N}$, and the fact that $|x'-y'|\leq 2n^{1/d}$, $\chemdist{u,N}(x',y') \leq |\ballZ(x',4n^{1/d})| \leq 9^d\cdot n$. 
This is precisely \eqref{eq:CxCy:n}. 
It is now immediate from \eqref{eq:CxCy:n} that 
\begin{equation}\label{eq:CxCy}
\text{for any $x,y\in\mathbb T_N$,~~ $\chemdist{u,N}(\mathcal C_x,\mathcal C_y)\leq 2\cdot (9^d + C)\cdot N$,}
\end{equation}
where $C=C(u)$ is the constant in \eqref{def:FuN}. 
Indeed, for any $x,y\in\mathbb T_N$, there exist $z_1,\dots, z_{m}\in\mathbb T_N$ such that $m=2\lfloor N/n\rfloor$, $z_1 = x$, $z_m = y$, and for all $1\leq i<m$, $|z_i - z_{i+1}| < n$. 
Applying \eqref{eq:CxCy:n} to each pair $z_i,z_{i+1}$ and using the definition of $\mathcal C_{z_i}$, we obtain from the triangle inequality that 
\[
\chemdist{u,N}(\mathcal C_x,\mathcal C_y) \leq \sum_{i=1}^{m-1}\chemdist{u,N}(\mathcal C_{z_i},\mathcal C_{z_{i+1}}) + (m-2)\cdot C\cdot n \leq 2\cdot (9^d + C)\cdot N ,\
\]
where $C=C(u)$ is the constant in \eqref{def:FuN}. 

It follows, in particular, from \eqref{eq:CxCy} that all the $\mathcal C_z$, $z\in \mathbb T_N$, are connected in $\mathcal V^{u,N}$. 
The event on the right-hand side of \eqref{eq:inclusion} now follows from \eqref{eq:yCx}, \eqref{eq:CxCy}, and the triangle inequality, with 
$C' = 4\cdot (9^d + C)$. 
\end{proof}
\begin{corollary}\label{cor:Cumax}
Let $\mathcal C^u_{max}$ be a connected component in $\mathcal V^{u,N}$ with the largest volume. 
It is immediate that if $\mathcal V^{u,N}_{n^{1/d}}$ is connected and has diameter greater than $n$, then 
it is the connected component in $\mathcal V^{u,N}$ with the largest volume, i.e., 
$\mathcal C^u_{max} = \mathcal V^u_{n^{1/d}}$.
Therefore, \eqref{eq:inclusion} implies that 
\begin{equation}\label{eq:inclusion:Cumax}
\mathcal F^{u,N}\cap \mathcal G^{u,N} \subseteq 
\left\{
\begin{array}{c}
\text{for all $x,y\in\mathcal C^u_{max}$, $\chemdist{u,N}(x,y)\leq C'N$, and}\\
\text{for all $x\in\mathbb T_N$, $\mathcal C^u_{max}\cap\ballZ(x,n^{1/d}) \neq\emptyset$}
\end{array}
\right\} ,\
\end{equation}
with $C'$ as in \eqref{eq:inclusion}.
\end{corollary}

\medskip

Now we have all the ingredients to prove Theorem~\ref{thm:chemdist:torus}. 
\begin{proof}[Proof of Theorem~\ref{thm:chemdist:torus}]
In \cite[Definition 2.4]{teixeira_windisch} a notion of strongly supercritical values of $u$ was introduced, 
which is similar in spirit to assumption \s{} for $\mathcal V^u$, but stronger. 
In particular, it is not yet known that strongly supercritical $u$'s exist for $d\in\{3,4\}$. 
We will not give this definition here, but just mention that 
analogously to \cite[(2.20)-(2.22)]{teixeira_windisch} one obtains that for any 
strongly supercritical $u$ and for any $\kappa>0$, 
\begin{equation}\label{eq:GuN}
\lim_{N\to\infty} N^\kappa \left(1- \mathbb{P}^{u,N}\left[\mathcal G^{u,N}\right]\right) = 0 .\
\end{equation}
In particular, we conclude from \eqref{eq:FuN}, \eqref{eq:inclusion:Cumax}, and \eqref{eq:GuN} that 
for any $u<\overline u$ which is strongly supercritical, \eqref{eq:chemdist:torus} holds. 
It remains to say that by \cite[Theorems 3.2 and 3.3]{Teixeira} (see \cite[(1.11)]{teixeira_windisch}), 
\begin{equation*}
 \text{for all $d \geq 5$,
 there exists
$\widehat u>0$ such that every $u \in (0, \widehat u)$ is strongly supercritical.}
\end{equation*}
This completes the proof of Theorem~\ref{thm:chemdist:torus} with $\widetilde u = \min(\overline u, \widehat u)$. 
\end{proof}

\section{Renormalization}\label{section:renorm}

In this short section we set up a general multi-scale renormalization scheme and state Theorem \ref{thm:badseed:proba}.

\medskip

Let $l_k,r_k,L_k$, $k\geq 0$ be sequences of positive integers such that 

\begin{equation}\label{def:Lk}
l_k>r_k, \qquad L_k = l_{k-1}\cdot L_{k-1}, \quad k \geq 1.
\end{equation}

For $k\geq 0$, we introduce the renormalized lattice graph $\GG_k$ by
\begin{equation*}
\GG_k = L_k \Z^d = \{L_kx ~:~ x\in\Z^d\} ,\
\end{equation*}
with edges between any pair of ($l^1$-)nearest neighbor vertices of $\GG_k$. 
For $k\geq 1$ and $x\in\GG_{k}$, let 
\begin{equation}\label{def:Lambdaxk}
\Lambda_{x,k} = \GG_{k-1}\cap(x+[0,L_k)^d) . \qquad
\left(\text{Note that $|\Lambda_{x,k}| = (l_{k-1})^d$.}\right) 
\end{equation}

\medskip

For $x\in\GG_0$, let $\badseed_x = \badseed_{x,0} = \badseed_{x,0,L_0}$ be a $\sigma(\Psi_y, y \in x + [-L_0, 3L_0)^d)$-measurable event. 
We call events of the form $\badseed_{x,0,L_0}$ \emph{seed events}, 
and we denote the family of seed events $(\badseed_{x,0,L_0}~:~L_0\geq 1,x\in\GG_0)$ by $\badseed$. 
For $k\geq 1$ and $x\in \GG_k$, we recursively define the events $\badseed_{x,k} = \badseed_{x,k,L_0}$ by
\begin{equation}\label{def:cascading}
\badseed_{x,k}= \bigcup_{  x_1,x_2\in \Lambda_{x,k} ; \, |x_1-x_2|_\infty > r_{k-1}\cdot L_{k-1} }
\badseed_{x_1,k-1} \cap \badseed_{x_2,k-1} ~.\
\end{equation}
(For simplicity, we omit the dependence of $\badseed_{x,k}$ on $L_0$ from the notation.)
The reader should think of the event $\badseed_{x,k}$ as ``unfavorable".
Informally,  \eqref{def:cascading} states that an unfavorable event inside a box of scale $L_k$ implies 
the occurrence of two unfavorable events inside sub-boxes of scale $L_{k-1}$ that are sufficiently far apart from each other.
By induction on $k$, the event $\badseed_{x,k}$ is $\sigma(\Psi_y, y \in x + [-L_0,L_k+2L_0)^d)$-measurable. 

\medskip

The following theorem 
helps us to give a good upper bound on the $\mathbb P^u$-probability
of some large-scale unfavorable events  
 $\badseed_{x,k}, x \in \GG_k$ in the case when the events
 $\badseed_x, x \in \GG_0$, are $\mathbb P^{u'}$-unlikely for any $u'$ in a small neighborhood of $u$.
\begin{theorem}\label{thm:badseed:proba}
Assume that the probability measures $\mathbb P^u$, $u\in(a,b)$, satisfy conditions \pp{} and \ppp{}. 
Let $u\in(a,b)$ and $\delta\in(0,1)$. 
Let 
\begin{equation}\label{def:scales}
\scexp = \lceil 1/\epsP \rceil,\qquad l_k = l_0\cdot 4^{k^\scexp},\qquad r_k = r_0\cdot 2^{k^\scexp},\quad k\geq 0,
\end{equation}
where $\epsP$ is defined in \ppp{}.
Let $\badseed_x,$ $x \in \GG_0,$ be all increasing events and $u' = (1+\delta)u$, 
or all decreasing events and $u' = (1-\delta)u$. If $u'\in(a,b)$ and 
\begin{equation}\label{eq:decoupling:condition}
\liminf_{L_0\to\infty} \sup_{x\in \GG_0} \mathbb P^{u'}\left[\badseed_x\right] = 0 ,\
\end{equation}
then there exist 
$C = C(\delta)<\infty$ and $L_0$ such that for any choice of $l_0,r_0\geq C$ and for all $k\geq 0$, 
\begin{equation}\label{eq:decoupling:result}
\sup_{x\in\GG_k} \mathbb P^u\left[\badseed_{x,k}\right] \leq 2^{-2^k} .\
\end{equation}
Moreover, if the limit in \eqref{eq:decoupling:condition} (as $L_0\to\infty$) exists and equals to $0$, then 
there exists 
$C' = C'(\delta,\badseed,l_0)<\infty$ such that 
the inequality \eqref{eq:decoupling:result} holds for all $l_0, r_0\geq C$, $L_0\geq C'$, and $k\geq 0$. 
\end{theorem}

\begin{remark}
The choice of $l_k$ and $r_k$ in \eqref{def:scales} is somewhat arbitrary and made to make some of the calculations more transparent, 
nevertheless, for the proof of Theorem~\ref{thm:badseed:proba}, we want that 
\begin{itemize}\itemsep0pt
\item[(a)] 
$\prod_{k=0}^{\infty} \left( 1+ r_k^{-\constP} \right)$ is close to $1$, where $\constP$ is defined in \eqref{eq:uwidehatu} (see \eqref{eq:condr0}), 
\item[(b)] 
$l_k$ grow fast enough (see \eqref{eq:condl0L0:b}, where the choice of $\scexp$ becomes clear), but 
\item[(c)]
$l_k$ do not grow too rapidly (see \eqref{eq:condl0L0:a}).
\end{itemize}

For the proof of Theorem~\ref{thm:Sinfty:chemdist}, we will also want that 
\begin{itemize}\itemsep0pt
\item[(d)] 
$\prod_{k=1}^{\infty} \left( 1+ \frac{r_{k-1}}{l_{k-1}} \right)<\infty$  (see \eqref{def_eq:N_double_prime}), 
\item[(e)] 
$l_0>4r_0$ (see Lemma~\ref{l:fromGktoGk-1}), and 
\item[(f)]
$l_k$ do not grow too fast (see \eqref{eq:Ls:lowerbound} and \eqref{eq:s:bounds}).
\end{itemize}
\end{remark}

\medskip

We postpone the proof of Theorem~\ref{thm:badseed:proba} until Section~\ref{sec:badseed:proof}. 
Our particular choice of the monotone seed events (and the verification of the condition
\eqref{eq:decoupling:condition} for these seed events) is the subject of Section \ref{section:conn_patterns}.

\bigskip

\section{Connectivity patterns in $\set$}
\label{section:conn_patterns}

The strategy for the proof of Theorem~\ref{thm:Sinfty:chemdist} is to define a $\mathbb P^u$-likely event $\mathcal H$, which implies 
the event on the left-hand side of \eqref{eq:Sinfty:chemdist}. 
The definition of $\mathcal H$ in Section~\ref{section:proof_of_chemdist_thm} will be based on a certain coarse-graining of $\set$. 
We will partition the vertices of the coarse-grained lattice $\GG_0$ into $0$-good and $0$-bad in such a way that 
\begin{enumerate}[(a)] 
\item if $x$ is a $0$-good vertex of $\GG_0$ then
 we can identify a macroscopic connected component $\mathcal C_x$ of $\set \cap (x + [0,L_0)^d)$ with the property that
 components $\mathcal{C}_x$, $\mathcal{C}_y$ corresponding to adjacent good vertices $x,y \in \GG_0$ are locally connected to each other in $\set$, and  
\item the  subgraph of $\GG_0$ spanned by good vertices contains a dense network of ``highways", i.e., nearest neighbor paths whose length
is comparable to the graph distance between their endpoints in $\GG_0$.  
\end{enumerate}
The crucial role in the definition of the event $\mathcal H$ will be played by Theorem~\ref{thm:badseed:proba}. 
Since Theorem~\ref{thm:badseed:proba} is only about increasing or decreasing events, 
we will define the event that $x\in\GG_0$ is $0$-good as an intersection of 
an increasing event $A^u_x$ and a decreasing event $B^u_x$. 
The event $A^u_x$ roughly asserts that the $L_0$-neighborhood of $x$ has a macroscopic connected component of $\set$ (see Section \ref{subsection:Aux}),
and the event $B^u_x$ roughly asserts that the $L_0$-neighborhood of $x$ has at most one macroscopic connected component of $\set$ 
(see Section \ref{subsection:Bux}).

The geometric details of the above listed plan  will be carried out in Section \ref{section:good_and_bad}.
The property  of $0$-good vertices outlined in (a) will be precisely formulated and proved in Lemma~\ref{l:fromG0toZd}. 
We will obtain the construction of the  ``highways" mentioned in (b) by iteratively applying  the construction of Lemma~\ref{l:fromGktoGk-1}.

\medskip

Before we proceed with the definition and properties of the events $A^u_x$ and $B^u_x$, observe that by 
\eqref{def:etau} and \eqref{def:setr}, 
$\mathbb P^u [0\in \set_r]\geq \eta(u)$ and $\mathbb P^u [0\in\set_r] \to \eta(u)$ as $r\to\infty$.
Therefore, if the measures $\mathbb P^u$, $u\in(a,b)$, satisfy condition \p{}, 
then by an appropriate ergodic theorem (see, e.g., \cite[Theorem~VIII.6.9]{DunfordSchwartz}), 
we get for any $u\in(a,b)$, 
\begin{equation}\label{eq:etaergodicity}
\lim_{L \to \infty} \frac{1}{L^d} 
\sum_{ x \in [0,L)^d} \mathds{1}\left( \,  x \in\set_L \, \right)
    \; \stackrel{\mathbb{P}^u\text{-a.s.}}{=} \; 
\lim_{L \to \infty} \frac{1}{L^d} 
\sum_{ x \in [0,L)^d} \mathds{1}\left( \,  x \in\set_\infty \, \right)
    \; \stackrel{\mathbb{P}^u\text{-a.s.}}{=} \; \eta(u),
\end{equation}
 where the $\{0,1\}$-valued random variable
 $\mathds{1}( A )$ denotes the indicator  of the event $A \in \mathcal F$.

\subsection{The increasing events $A^u_x$}\label{subsection:Aux}

\begin{definition}\label{def:Axu}
For $x\in \GG_0$ and $u>0$, let $A^u_x\in\mathcal F$ be the event that
\begin{itemize}\itemsep0pt
\item[(a)]
for each $e\in\{0,1\}^d$, the set $\set_{L_0}\cap(x+eL_0 + [0,L_0)^d)$
contains a connected component with at least $\frac 34 \eta(u) L_0^d$ vertices,
\item[(b)]
all of these $2^d$ components are connected in 
$\set \cap(x+[0,2L_0)^d)$.
\end{itemize}
\end{definition}
Note that for each $y'\in\Z^d$, 
the event 
$\left\{\, y' \in \set_{L_0} \,  \right\}$
is increasing and measurable with respect to $\sigma(\Psi_y~:~y\in \ballZ(y',L_0))$, thus
 $A^u_x$ is increasing and measurable with respect to $\sigma(\Psi_y~:~y\in x+[-L_0,3L_0)^d)$. 

\bigskip

For $u>0$ and $x\in \GG_0$, let 
\[
\overline A^u_{x,0} = (A^u_{x})^c ,\
\]
and for $u>0$, $k\geq 1$, and $x\in\GG_k$ let  
\begin{equation}\label{def:Auxk}
\overline A^u_{x,k} = 
\bigcup_{x_1,x_2\in \Lambda_{x,k} ; \, |x_1-x_2|_\infty > r_{k-1} \cdot L_{k-1}}
  \overline A^u_{x_1,k-1} \cap \overline A^u_{x_2,k-1}   \; .\
\end{equation}
The events $\overline A^u_{x,k}$ are decreasing and measurable with respect to $\sigma(\Psi_y,~ y\in x+[-L_0,L_k+2L_0)^d)$. 

\begin{lemma}\label{l:Auxk}
Assume that the measures $\mathbb P^u$, $u\in(a,b)$, satisfy conditions \p{} -- \ppp{} and \s{} -- \sss{}. 
Let $l_k$, $r_k$, and $L_k$ be defined as in \eqref{def:Lk} and \eqref{def:scales}. 
Then for each $u\in(a,b)$, there exist $C = C(u)<\infty$ and $C' = C'(u,l_0)<\infty$ such that for all 
$l_0,r_0\geq C$, $L_0\geq C'$, and $k\geq 0$, 
\[
\mathbb P^u\left[\overline A^u_{0,k}\right] \leq 2^{-2^k} .\
\]
\end{lemma}
\begin{proof}[Proof of Lemma~\ref{l:Auxk}]
It suffices to show that for any $u\in(a,b)$ there exists $\delta = \delta(u)>0$ such that $(1-\delta)u>a$ and 
\begin{equation}\label{eq:Axu}
\mathbb P^{(1-\delta)u}\left[A_0^u\right] \to 1 ,\quad \mbox{as} \quad L_0\to\infty .\
\end{equation}
Indeed, once \eqref{eq:Axu} is shown, 
the result will follow from Theorem~\ref{thm:badseed:proba} applied to the 
decreasing seed events $\badseed_x = \overline A^u_{x,0},\, x \in \GG_0$.

Let $u\in(a,b)$. 
By the condition \sss{}, we can choose $\varepsilon >0$ and $\delta=\delta(u)>0$ so that if we define $u' = (1-\delta)u$, 
then we have
\begin{equation}\label{eq:epsilondelta}
u' > a \quad\mbox{and}\quad
(1-4\varepsilon)^d\eta\left(u'\right) > \frac34 \eta(u) .\
\end{equation}
With such a choice of $\varepsilon$ and $\delta$, for $L_0\geq 1$, we obtain 
\[
\eta(u') (L_0 - 4\lfloor \varepsilon L_0\rfloor)^d >
 \frac34 \eta(u) L_0^d .\
\]
We consider the boxes 
\[
Q_e=eL_0 + [\,2\lfloor \varepsilon L_0\rfloor,L_0-2\lfloor\varepsilon L_0\rfloor\,)^d, \qquad e\in\{0,1\}^d.
\]
The volume of $Q_e$ is $|Q_e|=(L_0 - 4\lfloor \varepsilon L_0\rfloor)^d$.
 Using \eqref{eq:etaergodicity} and \eqref{eq:epsilondelta}, we get that 
with $\mathbb P^{u'}$-probability tending to $1$ as $L_0\to\infty$, each of the boxes $Q_e$, $e\in\{0,1\}^d$ contains at least
$\frac34 \eta(u) L_0^d$ vertices of $\set_{L_0}$.

By \eqref{eq:C2:epsilon}, 
we obtain that with $\mathbb P^{u'}$-probability tending to $1$ as $L_0\to\infty$, 
for all $e\in \{0,1\}^d$, $\set_{L_0}\cap Q_e$ is a subset of a connected component 
$\widetilde{\mathcal C_e}$ of 
\[
\set_{L_0}\cap\left(eL_0 + [\,\lfloor \varepsilon L_0\rfloor,L_0-\lfloor\varepsilon L_0\rfloor\,)^d\right) .
\]
By what is proved earlier, with $\mathbb P^{u'}$-probability tending to $1$ as $L_0\to\infty$, for all $e\in\{0,1\}^d$, we have 
$|\widetilde{\mathcal C_e}| \geq \frac34 \eta(u) L_0^d$. 

Moreover, again by \eqref{eq:C2:epsilon}, with $\mathbb P^{u'}$-probability tending to $1$ as $L_0\to\infty$, 
all $\widetilde{\mathcal C_e}$, $e\in\{0,1\}^d$, are in the same connected component of 
$\set_{L_0}\cap [0,2L_0)^d$. 

By Definition~\ref{def:Axu}, this implies \eqref{eq:Axu} and completes the proof of Lemma~\ref{l:Auxk}. 
\end{proof}

\subsection{The decreasing events $B^u_x$}\label{subsection:Bux}

\begin{definition}\label{def:Bux}
For $x\in \GG_0$ and $u>0$, let $B^u_x\in\mathcal F$ be the event that
for all $e\in\{0,1\}^d$, 
\[
\left|\set_{L_0}\cap(x+eL_0 + [0,L_0)^d)\right| \leq \frac54\eta(u) L_0^d .\
\]
\end{definition}
Note that the event $B^u_x$ is decreasing and measurable with respect to $\sigma(\Psi_y~:~y\in x+[-L_0,3L_0)^d)$. 

\bigskip

For $u>0$ and $x\in \GG_0$, let 
\[
\overline B^u_{x,0} = (B^u_{x})^c ,\
\]
and for $u>0$, $k\geq 1$, and $x\in\GG_k$ let  
\begin{equation}\label{def:Buxk}
\overline B^u_{x,k} = 
\bigcup_{x_1,x_2\in \Lambda_{x,k} ; \, |x_1-x_2|_\infty > r_{k-1} \cdot L_{k-1}}
  \overline B^u_{x_1,k-1} \cap \overline B^u_{x_2,k-1}   \; .\
\end{equation}
The events $\overline B^u_{x,k}$ are increasing and measurable with respect to $\sigma(\Psi_y,~ y\in x+[-L_0,L_k+2L_0)^d)$.

\begin{lemma}\label{l:Buxk} 
Assume that the measures $\mathbb P^u$, $u\in(a,b)$, satisfy conditions \p{} -- \ppp{} and \sss{}.
Let $l_k$, $r_k$, and $L_k$ be defined as in \eqref{def:Lk} and \eqref{def:scales}. 
For each $u\in(a,b)$, there exist $C = C(u)<\infty$ and $C' = C'(u,l_0)<\infty$ such that for all 
$l_0,r_0\geq C$, $L_0\geq C'$, and $k\geq 0$, 
\[
\mathbb P^u\left[\overline B^u_{0,k}\right] \leq 2^{-2^k} .\
\]
\end{lemma}
\begin{proof}[Proof of Lemma~\ref{l:Buxk}]
It suffices to show that for any $u\in(a,b)$ there exists $\delta=\delta(u)>0$ such that $(1+\delta)u<b$ and 
\begin{equation}\label{eq:Bxu}
\mathbb P^{(1+\delta)u}\left[B_0^u\right] \to 1 ,\quad \mbox{as} \quad L_0\to\infty .\
\end{equation}
Indeed, once \eqref{eq:Bxu} is shown, the result will 
follow from Theorem~\ref{thm:badseed:proba} applied to the
increasing seed events $\badseed_x = \overline B^u_{x,0},\, x \in \GG_0$.

Let $u\in(a,b)$. 
By the condition \sss{}, we can choose $\delta=\delta(u)>0$ so that with $u' = (1+\delta)u$, we have 
\begin{equation}\label{eq:delta54}
u' < b\quad\mbox{and}\quad\eta\left(u'\right) < \frac{5}{4} \eta(u) .\
\end{equation}
Therefore, \eqref{eq:etaergodicity} and \eqref{eq:delta54} imply that, with $\mathbb P^{u'}$ probability tending to $1$ as $L_0\to\infty$, 
for all $e\in\{0,1\}^d$, 
\[
\left|\set_{L_0}\cap\left(eL_0 + [0,L_0)^d\right)\right| < \frac54\eta(u) L_0^d .\
\]
By Definition~\ref{def:Bux}, this is precisely \eqref{eq:Bxu}. 
The proof of Lemma~\ref{l:Buxk} is complete.
\end{proof}

\bigskip

\section{Good and bad vertices}
\label{section:good_and_bad}

The aim of this section is to  provide the geometric details of one iteration of the recursive construction
 of short paths in $\set$. We follow the strategy outlined in the introduction of Section \ref{section:conn_patterns}.
Throughout this section, $l_k$, $r_k$, and $L_k$, $k\geq 1$, are positive integers satisfying \eqref{def:Lk}.

\begin{definition}\label{def:good}
Let $u\in(a,b)$. For $k\geq 0$, we say that $x\in\GG_k$ is $k$-{\it bad} if the event 
\[
\overline A^u_{x,k}\cup \overline B^u_{x,k}
\]
occurs. Otherwise, we say that $x$ is $k$-{\it good}. 
Note that the event $\{\text{$x$ is $k$-good}\}$ is measurable with respect to 
$\sigma(\Psi_y~:~ y\in x+ [-L_0, L_k + 2L_0)^d)$, see below \eqref{def:Auxk} and \eqref{def:Buxk}.
\end{definition}

In the remainder of this section we prove useful geometric properties of $k$-good vertices. 
All the results of this section are purely deterministic. 
We fix the parameter $u\in(a,b)$. 

\begin{lemma}\label{l:fromG0toZd}
Let $x$ and $y$ be nearest neighbors in $\GG_0$ and assume that they are both $0$-good. 
\begin{itemize}\itemsep0pt
 \item[(a)]
Each of the graphs $\set_{L_0}\cap(z + [0,L_0)^d)$, with $z\in \{x,y\}$, 
contains the unique connected component $\mathcal C_z$ with at least $\frac34 \eta(u) L_0^d$ vertices, and
\item[(b)] 
$\mathcal C_x$ and $\mathcal C_y$ are connected in the graph $\set\cap ((x+[0,2L_0)^d)\cup(y+[0,2L_0)^d))$. 
\end{itemize}
\end{lemma}
\begin{proof}[Proof of Lemma~\ref{l:fromG0toZd}]
Let $x$ and $y$ be nearest neighbors in $\GG_0$ and assume they are both $0$-good, i.e., the events 
$A^u_x$, $B^u_x$, $A^u_y$, and $B^u_y$ occur. 

By Definition~\ref{def:Axu}, since $A^u_x$ and $A^u_y$ occur, 
the graphs $\set_{L_0}\cap(x + [0,L_0)^d)$ and $\set_{L_0}\cap(y + [0,L_0)^d)$ 
contain connected components $\mathcal C_x$ and $\mathcal C_y$ with at least $\frac34 \eta(u) L_0^d$ vertices, 
which are connected in $\set\cap((x+[0,2L_0)^d)\cup(y+[0,2L_0)^d))$.  

It remains to show that $\mathcal C_x$ and $\mathcal C_y$ are unique. 
By Definition~\ref{def:Bux}, since $B^u_x$ and $B^u_y$ occur,
each of the sets $\set_{L_0}\cap(x + [0,L_0)^d)$ and $\set_{L_0}\cap(y + [0,L_0)^d)$ 
contain at most $\frac54 \eta(u) L_0^d$ vertices. 
Since $2\cdot \frac 34 > \frac 54$, there can be at most one connected component of size 
at least $\frac34 \eta(u) L_0^d$ in each of the graphs 
$\set_{L_0}\cap(x + [0,L_0)^d)$ and $\set_{L_0}\cap(y + [0,L_0)^d)$. 
This implies the uniqueness of $\mathcal C_x$ and $\mathcal C_y$ 
and finishes the proof of Lemma~\ref{l:fromG0toZd}. 
\end{proof}

\begin{lemma}\label{l:fromGktoGk-1}
Let $k\geq 1$. 
Let $\pi = (x_0,\ldots,x_m)$ be a nearest neighbor path in $\GG_k$ of $k$-good vertices. If $l_0> 4r_0$, 
then there exists a nearest neighbor path $\pi' = (x_0',\ldots,x_n')$ in $\GG_{k-1}$ of $(k-1)$-good vertices, such that 
\begin{itemize}\itemsep0pt
\item[(a)]
$x_0'\in\Lambda_{x_0,k}$, $x_n'\in\Lambda_{x_m,k}$ (recall \eqref{def:Lambdaxk}), and 
\item[(b)]
$n\leq (1+8\cdot\frac{r_{k-1}}{l_{k-1}})\cdot l_{k-1}m$. 
\end{itemize}
\end{lemma}
\begin{proof}[Proof of Lemma~\ref{l:fromGktoGk-1}]
Fix $k\geq 1$ and a nearest neighbor path $\pi = (x_0,\ldots,x_m)$ in $\GG_k$ of $k$-good vertices. 
For each $0\leq i\leq m$, since $x_i$ is $k$-good (see Definition~\ref{def:good}, \eqref{def:Auxk}, and \eqref{def:Buxk}),
there exist $a_i,b_i\in\Lambda_{x_i,k}$ such that 
all the vertices contained in 
\begin{equation}\label{def:Lambda*}
\Lambda^*_i = \Lambda_{x_i,k} \setminus \left((a_i +[0,r_{k-1} L_{k-1})^d)\cup(b_i +[0,r_{k-1} L_{k-1})^d)\right)
\end{equation}
are $(k-1)$-good. 
Since $d \geq 2$ and $l_{k-1} > r_{k-1}$ are integers, each of the sets $\Lambda^*_i$, $0\leq i\leq m$ is non-empty. 
In particular, this settles the case $m=0$, and from now on we may and will assume that $m\geq 1$. 

\medskip

Our strategy to construct the path $\pi'$ is as follows. 
We show that for any $0\leq i\leq m-1$, there exists a pair of vertices $y^*_i\in\Lambda^*_i$ and $z^*_{i+1}\in\Lambda^*_{i+1}$ 
such that 
\begin{itemize}\itemsep0pt
\item[(a)]
for $0\leq i\leq m-1$, $y^*_i$ is connected to $z^*_{i+1}$ by a nearest neighbor path in $\Lambda^*_i\cup\Lambda^*_{i+1}$ of length $l_{k-1}$,
\item[(b)]
for $1\leq i\leq m-1$, $z^*_i$ is connected to $y^*_i$ by a nearest neighbor path in $\Lambda^*_i$ of length at most $8r_{k-1}$. 
\end{itemize}
We then define $\pi'$ as the nearest neighbor path from $y^*_0\in\Lambda^*_1$ to $z^*_m\in\Lambda^*_m$ obtained by concatenating 
all the paths constructed in (a) and (b). 

\medskip

Let $0\leq i\leq m-1$. Let $\alpha_i$ be the smallest number in $\{1,\ldots,d\}$ such that $e_{\alpha_i}$ is orthogonal to $(x_{i+1} - x_i)$. 
By \eqref{def:Lambda*}, since $l_{k-1}>4r_{k-1}$, there exists an integer $0\leq j_i\leq 4r_{k-1}$ such that the hyperplane $H^*_i$ orthogonal to $e_{\alpha_i}$ 
and passing through $x_i + L_{k-1}j_ie_{\alpha_i}\in\Lambda_{x_i,k}$ (and $x_{i+1}+L_{k-1}j_ie_{\alpha_i}\in\Lambda_{x_{i+1},k}$) satisfies 
\[
H^*_i\cap(\Lambda_{x_i,k}\cup\Lambda_{x_{i+1},k}) \subset (\Lambda^*_i \cup\Lambda^*_{i+1}) .\
\]
We define 
\[
y^*_i := x_i + L_{k-1}j_ie_{\alpha_i}
\qquad
\mbox{and}
\qquad
z^*_{i+1} := x_{i+1}+L_{k-1}j_ie_{\alpha_i} .\
\]
Note that the line segment in $\GG_{k-1}$ from $y^*_i$ to $z^*_{i+1}$ has length $l_{k-1}$ and is contained in $\Lambda^*_i\cup\Lambda^*_{i+1}$. 
Therefore, it remains to show that $y^*_i$ and $z^*_i$ satisfy the requirement (b) above.

\medskip

Let $1\leq i\leq m-1$. We consider separately the cases $\alpha_i \neq \alpha_{i-1}$ and $\alpha_i = \alpha_{i-1}$ 
(see Figures \ref{fig:alphaiineq} and \ref{fig:alphaiequality}, respectively).

\begin{figure}[t!]
\begin{center}
\psfragscanon
\includegraphics[height=8cm]{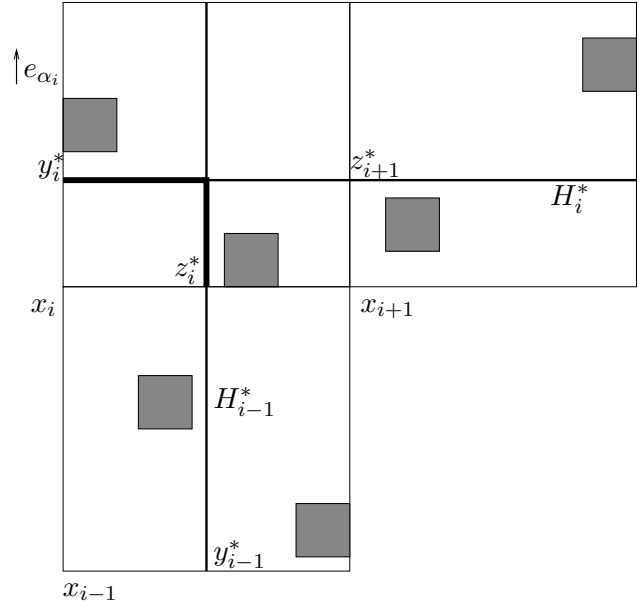}
\caption{The case $d=2$ and $\alpha_i \neq \alpha_{i-1}$.}\label{fig:alphaiineq}
\end{center}
\end{figure}

If $\alpha_i\neq\alpha_{i-1}$, then 
\[
y^*_i + L_{k-1}j_{i-1}e_{\alpha_{i-1}} = z^*_i + L_{k-1} j_i e_{\alpha_i} \in H^*_i\cap H^*_{i-1} .\
\]
Note that the line segment from $y^*_i$ to $y^*_i + L_{k-1}j_{i-1}e_{\alpha_{i-1}}$ in $\GG_{k-1}$ 
is contained in $\Lambda_{x_i,k}\cap H^*_i \subset \Lambda^*_i$, and 
the line segment from $z^*_i$ to $z^*_i + L_{k-1} j_i e_{\alpha_i}$ in $\GG_{k-1}$
is contained in $\Lambda_{x_i,k}\cap H^*_{i-1} \subset \Lambda^*_i$. 
We conclude that $y^*_i$ and $z^*_i$ are connected in $\Lambda^*_i$ by a nearest neighbor path in $\GG_{k-1}$ of length $j_{i-1} + j_i \leq 8r_{k-1}$. 

\begin{figure}[t!]
\begin{center}
\psfragscanon
\includegraphics[height=5cm]{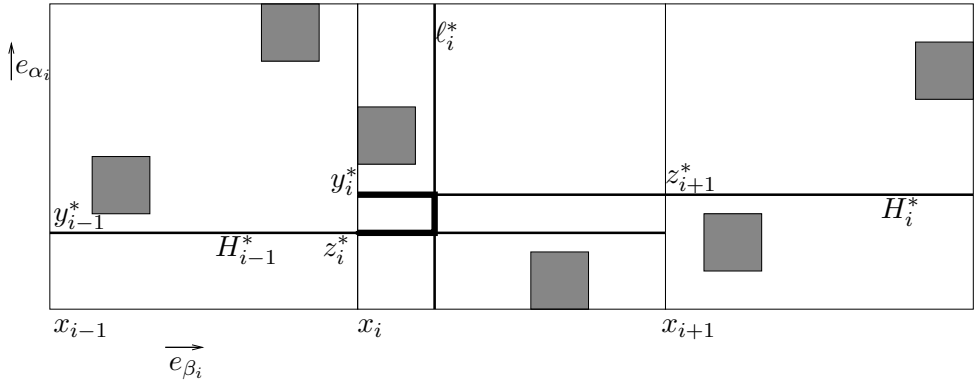}
\caption{The case $d=2$ and $\alpha_i = \alpha_{i-1}$.}\label{fig:alphaiequality}
\end{center}
\end{figure}

If $\alpha_i = \alpha_{i-1}$, let $\beta_i\in\{1,\ldots,d\}$ be such that the unit vector $e_{\beta_i}$ is parallel to $(x_i-x_{i-1})$. 
By the definitions of $\alpha_i$ and $\beta_i$, $\alpha_i\neq\beta_i$.   
By \eqref{def:Lambda*}, since $l_{k-1}>2r_{k-1}$, there exists an integer $0\leq k_i\leq 2r_{k-1}$ such that 
the line $\ell^*_i$ parallel to $e_{\alpha_i}$ and passing through $y^*_i + L_{k-1} k_i e_{\beta_i}$ (and $z^*_i + L_{k-1} k_i e_{\beta_i}$) satisfies
\[
\ell^*_i \cap \Lambda_{x_i,k} \subset \Lambda^*_i .\
\]
Note that the line segment from $y^*_i$ to $y^*_i + L_{k-1}k_ie_{\beta_i}$ in $\GG_{k-1}$ 
is contained in $\Lambda_{x_i,k}\cap H^*_i \subset \Lambda^*_i$,  
the line segment from $z^*_i$ to $z^*_i + L_{k-1}k_ie_{\beta_i}$ in $\GG_{k-1}$
is contained in $\Lambda_{x_i,k}\cap H^*_{i-1} \subset \Lambda^*_i$, and 
the line segment from $y^*_i + L_{k-1}k_ie_{\beta_i}$ to $z^*_i + L_{k-1}k_ie_{\beta_i}$ in $\GG_{k-1}$
is contained in $\Lambda_{x_i,k}\cap \ell^*_i \subset \Lambda^*_i$. 
We conclude that $y^*_i$ and $z^*_i$ are connected in $\Lambda^*_i$ by a nearest neighbor path in $\GG_{k-1}$ 
of length $2k_i + |j_{i-1} - j_i| \leq 8r_{k-1}$.

We have shown the existence of vertices $y^*_i$ and $z^*_i$ satisfying (a) and (b) above. 
The proof of Lemma~\ref{l:fromGktoGk-1} is complete.
\end{proof}

\bigskip

\section{Proof of Theorem~\ref{thm:Sinfty:chemdist}}
\label{section:proof_of_chemdist_thm}

In our proof of Theorem~\ref{thm:Sinfty:chemdist} we will follow the strategy sketched at the beginning of Section~\ref{section:conn_patterns}. 
We will define an event $\mathcal H$ using the definition of $k$-good vertices from Section~\ref{section:good_and_bad}. 
Using the seed estimates of Section \ref{section:conn_patterns}, 
we will show that $\mathcal H$ occurs with high $\mathbb P^u$-probability. 
We will then apply the geometric results of Section \ref{section:good_and_bad} to 
show that $\mathcal H$ implies the event on the left-hand side of \eqref{eq:Sinfty:chemdist}. 
This part of the proof is purely deterministic. 

\medskip

Let $u\in(a,b)$. 
It suffices to show that for some $c=c(u)>0$ and $C = C(u)<\infty$, the inequality
 \eqref{eq:Sinfty:chemdist} holds for all $R\geq C'$ with large enough $C' = C'(u)$.  
Increasing $C$ possibly even further then yields the validity of \eqref{eq:Sinfty:chemdist} for all $R \ge 1$.

We take $l_0>4r_0$ and $L_0$ so that the implications of Lemmas \ref{l:Auxk} and \ref{l:Buxk} hold 
(with $l_k$ and $r_k$ as in \eqref{def:scales} and $L_k$ as in \eqref{def:Lk}). 
Let $s$ be the largest integer such that $L_s \leq R^{1/d}$, i.e.,
\begin{equation}\label{def:s}
s=\max \{ \, s' \; : \;   L_{s'} \leq R^{1/d} \} .\
\end{equation}
We assume that $R\geq L_0^d$, so that $s$ is well-defined. 
Let $\mathcal H$ be the event that 
\begin{itemize}\itemsep0pt
\item[(a)]
each $z\in\GG_s\cap\ballZ(0,2R)$ is $s$-good, and
\item[(b)]
for any $z\in\GG_s\cap\ballZ(0,2R)$ and $x,y\in\set_{L_s}\cap(z+[0,2L_s)^d)$, 
$x$ is connected to $y$ by a path in $\set$ of length at most $(4 L_s)^d$. 
\end{itemize}
The event in part (b) of the definition of $\mathcal H$ is implied by the event that 
\begin{itemize}\itemsep0pt
\item[(b')]
for any $z\in\GG_s\cap\ballZ(0,2R)$ and $x,y\in\set_{L_s}\cap(z+[0,2L_s)^d)$, 
$x$ is connected to $y$ by a simple path in $\set\cap(z+[-L_s,3L_s)^d)$. 
\end{itemize}
By Definition~\ref{def:good}, \eqref{eq:C2:epsilon}, and Lemmas \ref{l:Auxk} and \ref{l:Buxk}, 
for any $u\in(a,b)$ and some constants $c = c(u)>0$ and $C = C(u)<\infty$, 
\[
\mathbb P^u\left[\mathcal H^c \right]\leq 
|\GG_s\cap\ballZ(0,2R)|\cdot 2\cdot 2^{-2^s} + |\GG_s\cap\ballZ(0,2R)|\cdot C\cdot e^{-c\funcS(u,L_s)} ~.\
\]
By \eqref{def:Lk}, \eqref{def:scales}, and \eqref{def:s}, for all $R\geq L_0^d$, 
\[
R^{1/d} \leq L_{s+1} = l_s\cdot L_s \leq l_0\cdot 4 \cdot (L_s)^{1+2^\scexp},
\]
which implies that 
\begin{equation}\label{eq:Ls:lowerbound}
L_s \geq \frac{1}{4l_0} R^{1/d(1+2^\scexp)} .\
\end{equation}
Using \eqref{eq:funcS}, \eqref{def:Lk}, \eqref{def:scales}, and \eqref{eq:Ls:lowerbound}, we deduce that there exist $c' = c'(u)>0$ and $C' = C'(u)<\infty$ such that for all $R\geq C'(u)$, 
\begin{equation}\label{eq:s:bounds}
2^s \geq (\binlog R)^{1+\constS} \quad\mbox{and}\quad \funcS(u,L_s) \geq c'(\binlog R)^{1+\constS} ~.\
\end{equation}
Note that we also have $|\GG_s\cap\ballZ(0,2R)| \leq (4R+1)^d$.
Therefore, for any $u\in(a,b)$, there exist $c=c(u)>0$ and $C = C(u)<\infty$ such that for all $R\geq C(u)$,  
\begin{equation}\label{eq:eventH:proba}
\mathbb P^u\left[\mathcal H \right]\geq 1 - Ce^{-c(\binlog R)^{1+\constS}} .\
\end{equation}

\medskip

We will now show that the occurrence of $\mathcal H$ implies the event on the left-hand side of \eqref{eq:Sinfty:chemdist}. 
Then \eqref{eq:Sinfty:chemdist} will immediately follow from \eqref{eq:eventH:proba}. 

Assume that $\mathcal H$ occurs. 
Take $x,y\in\set_R\cap\ballZ(0,R)$. 
We will show that $\distS(x,y)\leq CR$ for some $C = C(u)<\infty$. 

If there exists $z\in\GG_s\cap\ballZ(0,2R)$ such that $x,y\in (z+[0,2L_s)^d)$, 
then the result follows from part (b) of the definition of $\mathcal H$ and \eqref{def:s}. 

Therefore, we assume that there is no such $z$. 
There exist unique vertices $x',y'\in\GG_s\cap\ballZ(0,2R)$ such that 
\[
x\in \set_R\cap(x' + [0,L_s)^d),\qquad y\in \set_R\cap(y' + [0,L_s)^d) .\ 
\]
Moreover, $|x'-y'|_\infty\geq 2L_s$. 
By part (a) of the definition of $\mathcal H$, there exists a nearest neighbor path
of $s$-good vertices in $\GG_s\cap\ballZ(0,2R)$ connecting $x'$ and $y'$ of length  
\[
N' \leq d \lceil (4R+1)/L_s\rceil .\
\] 

Applying Lemma~\ref{l:fromGktoGk-1} repeatedly for $k=s, s-1, \ldots, 1,$
we deduce that there exists $0$-good vertices $x''\in \GG_0\cap(x' + [0,L_s)^d)$ and $y''\in \GG_0\cap(y' + [0,L_s)^d)$ 
which are connected by a nearest neighbor path of $0$-good vertices in $\GG_0$ of length 
\begin{equation}\label{def_eq:N_double_prime}
N'' \leq \prod_{k=1}^s\left(1 + 8\cdot \frac{r_{k-1}}{l_{k-1}}\right)\cdot \frac{L_s}{L_0}\cdot N' .\
\end{equation}
By Lemma~\ref{l:fromG0toZd}, there exist 
\[
x^*\in\set\cap(x'' + [0,L_0)^d),\qquad
y^*\in\set\cap(y'' + [0,L_0)^d) 
\]
such that 
$x^*$ is connected to $y^*$ by a path in $\set$ of length 
\[
N^* \leq (2L_0)^d\cdot N''.
\]
Note that 
\[
x^*\in\set\cap(x' + [0,L_s)^d),\qquad
y^*\in\set\cap(y' + [0,L_s)^d) .\
\]
Since $|x'-y'|_\infty \geq 2L_s$, we have $|x^* - y^*|_\infty > L_s$. 
In particular, 
\[
x^*,y^*\in\set_{L_s} .\
\]
Since  $x,y\in\set_R\cap\ballZ(0,R) \subseteq \set_{L_s}\cap\ballZ(0,R)$, 
$x,x^*\in(x' + [0,L_s)^d)$, and $y,y^*\in(y' + [0,L_s)^d)$, 
we obtain by part (b) of the definition of $\mathcal H$ that 
$\distS(x,x^*)\leq (4 L_s)^d$ and $\distS(y,y^*)\leq (4 L_s)^d$.
As a result, by the triangle inequality, we have 
\[
\distS(x,y) \leq N^* + 2 (4 L_s)^d .\
\]
It remains to observe that there exists $C = C(u)<\infty$ such that $N^*\leq \frac12 C R$ and
 $2 (4 L_s)^d \leq \frac12 C R$.
Indeed, $N^*\leq \frac12 C R$ follows from \eqref{def:scales} and the definitions of $N'$, $N''$ and $N^*$, 
and $2 (4 L_s)^d \leq \frac12 C R$ follows from \eqref{def:s}. 

We have thus shown that if the event $\mathcal H$ occurs, then any $x,y\in\set_S\cap\ballZ(0,R)$ are connected in $\set$ and  
$\distS(x,y)\leq C R$. Theorem~\ref{thm:Sinfty:chemdist} now follows from \eqref{eq:eventH:proba}. 
\qed

\section{Proof of Theorem~\ref{thm:badseed:proba}}\label{sec:badseed:proof}

Assume that the family of probability measures $\mathbb P^u$, $u\in(a,b)$, satisfies \pp{} and \ppp{}.
The proof of Theorem~\ref{thm:badseed:proba} goes by induction on $k$. The key idea is that for a certain sequence $(u_k)_{k=1}^{\infty}$ of parameters the expression
 $\sup_{x \in \GG_{k+1}} \P^{u_{k+1}}\left[\badseed_{x,k+1}\right]$ can be upper bounded by something which is not much bigger than
$\sup_{x \in \GG_{k}} \P^{u_k}\left[\badseed_{x,k}\right]^2$, see \eqref{eq:DecIneqApp}.
Throughout this section we fix the parameters $\constP$ and $\epsP$ appearing in \eqref{eq:uwidehatu} and \eqref{def:epsP}, respectively. 

\medskip

We only give the proof in the case of increasing events $\badseed_{x}$. 
The case when $\badseed_{x}$ are decreasing can be treated similarly, and we omit it. 

\medskip

Fix $u\in(a,b)$. Let $\delta\in(0,1)$ be such that $u' = (1+\delta)u\in (a,b)$. 
Let the sequences $l_k$, $r_k$, and $L_k$ be defined as in \eqref{def:scales} and \eqref{def:Lk}. 
Take $r_0\geq \RP$ (the constant from the assumption \ppp{}) 
large enough so that 
\begin{equation}\label{eq:condr0}
\prod_{k=0}^\infty\left(1 + r_k^{-\constP}\right) \leq 1+\delta .\
\end{equation}
Let $\badseed_{x}$, $x\in\Z^d$, be increasing events such that $\badseed_{x}\in\sigma(\Psi_y,~y\in x + [-L_0,3L_0)^d)$. 
By \eqref{def:cascading}, for all $k\geq 0$ and $x\in\GG_k$,
\begin{equation}\label{eq:badseed:meas}
\text{ $\badseed_{x,k}$ is increasing, and 
$\badseed_{x,k}\in \sigma(\Psi_y,~y\in x + [-L_0,L_k + 2L_0)^d)$.}
\end{equation} 
Let 
\begin{equation}\label{eq:uk}
u_0 = u'\quad\mbox{and}\quad u_k = \left(1 + r_k^{-\constP}\right)\cdot u_{k+1},\quad k\geq 0 .\
\end{equation}
By \eqref{eq:condr0}, we have 
\begin{equation}\label{eq:uinftyu}
a< u \leq u_k< b,\quad\mbox{for all }~ k\geq 0 .\
\end{equation}

For $k\geq 0$, let 
\begin{equation}\label{eq:kappa}
\kappa_k = 1 + l_0\cdot\sum_{i=k+1}^\infty i^{-2} \quad \left(= \kappa_{k+1} + l_0\cdot(k+1)^{-2}\right).\
\end{equation}
By monotonicity of the events $\badseed_{x}$, \eqref{eq:uinftyu}, condition \pp{}, 
and the fact that $\kappa_k\geq 1$, for all $k$, in order to prove \eqref{eq:decoupling:result}, 
it suffices to show that 
\begin{equation}\label{eq:indStatement}
\sup_{x\in\GG_k} \mathbb P^{u_k}\left[\badseed_{x,k}\right] \leq 2^{-\kappa_k2^k}.
\end{equation}
We will prove \eqref{eq:indStatement} by induction on $k$. 

\medskip

{\em Induction start:} 
By \eqref{eq:decoupling:condition} and \eqref{eq:kappa}, for each $l_0$, there exist $L_0$ such that
\begin{equation} \label{eq:seedEst}
\sup_{x\in\GG_0} \mathbb P^{u_0}\left[\badseed_{x,0}\right] \leq 2^{-\kappa_0},
\end{equation}
 hence \eqref{eq:indStatement} holds for $k=0$ and such $L_0 = L_0(l_0,u_0,\badseed)$. 
(Later $l_0$ will be chosen large enough, independently of $L_0$.) 

\medskip

{\em Induction step:} We assume that \eqref{eq:indStatement} holds for $k\geq 0$, 
and now prove that it also holds for $k+1$. 
For each $x\in\GG_k$, we have 
\begin{align}
\P^{u_{k+1}}\left[\badseed_{x,k+1}\right]
&\stackrel{\eqref{def:cascading}}\le 
\sum_{x_1, x_2 \in \Lambda_{x,k+1} \; : \; \vert x_1 - x_2 \vert_\infty >  r_{k}\cdot L_{k}}
\P^{u_{k+1}} \left[ \badseed_{x_1,k}\cap \badseed_{x_2,k}\right] \nonumber\\
&\stackrel{\eqref{eq:decorrelation:increasing},\eqref{eq:badseed:meas},\eqref{eq:uk}}\le \vert \Lambda_{x,k+1} \vert^2
\left( \sup_{x\in\GG_k} \P^{u_k} \left[\badseed_{x,k}\right]^2 + e^{- \funcP(L_k)} \right)
 \label{eq:DecIneqApp}\\
&\stackrel{\eqref{def:Lambdaxk}}\le l_k^{2d} 
\left(2^{- \kappa_k 2^{k+1}} + e^{- \funcP(L_k)} \right) .\
\label{eq:indRHS}
\end{align}
Here, in \eqref{eq:DecIneqApp} we applied \eqref{eq:decorrelation:increasing} from the condition \ppp{} with $R = r_k$, $L = L_k$, $u = u_k$, 
and $\widehat u = u_{k+1}$ (the requirements of condition \ppp{} are satisfied 
by the choice of $r_0\geq \RP$, \eqref{eq:badseed:meas}, \eqref{eq:uk}, and \eqref{eq:uinftyu}), 
and \eqref{eq:indRHS} follows from the induction assumption. 

By the definitions of $l_k$ in \eqref{def:scales} and $\kappa_k$ in \eqref{eq:kappa}, 
there exists $C<\infty$ such that for all $k\geq 0$ and $l_0\geq C$,
\begin{equation}\label{eq:condl0L0:a}
\binlog(l_k^{2d}) - \kappa_k 2^{k+1} \leq - \kappa_{k+1} 2^{k+1} - 1 ,\
\end{equation}
and by the definitions of $\scexp$, $l_k$  in \eqref{def:scales}, $L_k$ in \eqref{def:Lk}, 
$\kappa_k$ in \eqref{eq:kappa}, $\epsP$ in \eqref{def:epsP}, and the fact that $\scexp\cdot\epsP \geq 1$,  
there exists $C' = C'(l_0)<\infty$ such that 
for all $k\geq 0$,  and $L_0 \geq C'$, 
\begin{equation}\label{eq:condl0L0:b}
\binlog(l_k^{2d}) - \funcP(L_k) \leq 
- \kappa_{k+1} 2^{k+1} - 1 .\
\end{equation}
Plugging \eqref{eq:condl0L0:a} and \eqref{eq:condl0L0:b} into \eqref{eq:indRHS}, 
we obtain that \eqref{eq:indStatement} holds for $k+1$. 
We finish the proof by specifying the order in which we choose large constants $r_0$, $l_0$ and $L_0$. 
We first choose $r_0\geq \RP$ (the constant from the condition \ppp{}) so that \eqref{eq:condr0} holds. 
We then choose $l_0$ so that \eqref{eq:condl0L0:a} holds.
In other words, there exists a finite constant $C = C(\delta) \geq \RP$ such that for all $l_0, r_0\geq C$, 
\eqref{eq:condr0} and \eqref{eq:condl0L0:a} hold. 
Finally, we choose $L_0$ to satisfy \eqref{eq:seedEst} and \eqref{eq:condl0L0:b}. 
In particular, if the limit in \eqref{eq:decoupling:condition} exists, then there exists 
$C' = C'(\delta,\badseed,l_0)<\infty$ such that for all $L_0\geq C'$, 
\eqref{eq:seedEst} and \eqref{eq:condl0L0:b} hold. 
We have thus finished the proof of Theorem~\ref{thm:badseed:proba} in the case of increasing events $\badseed_{x}$. 
\qed

\section{Proof of Theorem~\ref{thm:shapeThm}}\label{section:shape_proof}

We begin by briefly outlining our proof of Theorem \ref{thm:shapeThm}, which
 adapts the standard recipe for shape theorems (see, e.g.\ \cite{frog}) to our setting.

In Section \ref{subsection_extension_of_chemdist} we define the pseudometric $\distSZ(\cdot,\cdot)$ on $\Z^d$ satisfying 
$\distSZ(x,y) = \distS(x,y)$ for $x,y\in\set_\infty$. 
This trick will allow us to work with $\P^u$ rather than the more cumbersome 
$\P^u [ \cdot \, \vert \, 0 \in \set_\infty]$.
After stating some properties of $\distSZ(\cdot,\cdot)$, 
we rephrase Theorem \ref{thm:shapeThm} in 
terms of $\distSZ(\cdot,\cdot)$ as Proposition \ref{prop:shapeThm_ext}.

In Section \ref{subsection_proof_of_ext_shape_proposition} we prove Proposition \ref{prop:shapeThm_ext}.
We first use Kingman's subadditive ergodic theorem  \cite{kingman} to deduce in Lemma \ref{lemma_norm}
 that the large-scale behaviour of $\distSZ(0,\cdot)$ along half-lines can be $\mathbb P^u$-a.s. 
approximated by some norm $p_u(\cdot)$ on $\R^d$. 
The convex set $D_u$ will turn out to be the unit ball with respect to this norm. 
The proof of Proposition \ref{prop:shapeThm_ext} combines the result of Lemma \ref{lemma_norm} with
a covering argument in which Theorem \ref{thm:Sinfty:chemdist} serves as a preliminary bound.

\subsection{A suitable pseudometric on $\Z^d$}\label{subsection_extension_of_chemdist}

In this section we define the pseudometric $\distSZ(\cdot,\cdot)$ on $\Z^d$ satisfying 
$\distSZ(x,y) = \distS(x,y)$ for $x,y\in\set_\infty$. 
For that we assume that 
\begin{equation}\label{eq:set:nonempty}
\set_\infty\neq\emptyset .\
\end{equation}
We first
choose a bijection $\ell : \Z^d \to \N$ which satisfies
\[ \forall\, x,y \in \Z^d \; : \;  |x|_\infty < |y|_\infty \quad \implies \quad \ell(x)<\ell(y). \]
Thus $\ell$ gives us a labelling of $\Z^d$ such that vertices with smaller norm have smaller labels.
Given $x \in \Z^d$ and $\emptyset \neq V \subseteq \Z^d$ 
we  define $\Phi(x, V ) \in V$ to be one of the $l^{\infty}$-closest vertices of
 $V$ to $x$ by 
\[\Phi(x, V )=x^*, \qquad \text{where} \qquad  
 \ell(x^*-x)=\min \{ \, \ell(y-x) \; : \; y \in V \, \}.\]
Note that 
\begin{equation}\label{eq:Psi:xinV}
\Phi(x, V)=x\quad\text{for all} \quad x\in V,
\end{equation}
and
\begin{equation}\label{translation_inv_psi}
\Phi(x+y, V+y )=\Phi(x, V )+y \quad \text{for all} \quad x,y \in \Z^d.
\end{equation}
We define $\distSZ~:~\Z^d\times\Z^d \to \Z_+\cup\{\infty\}$ by the formula 
\begin{equation}\label{def:distSZ}
\distSZ(x,y)=
\distS \big(\Phi(x, \set_\infty  ), \, \Phi(y,\set_\infty )\big), \qquad
x,y \in \Z^d.
\end{equation}
By \eqref{eq:set:nonempty}, the function $\distSZ$ is well-defined. 
In fact, $\distSZ(\cdot,\cdot)$ is a pseudometric on $\Z^d$. 
By \eqref{eq:Psi:xinV}, 
\begin{equation}\label{chemdist_restriction}
\forall\,  x,y \in \set_\infty\; : \; \distSZ(x,y)=\distS(x,y) \geq |x-y|_1.
\end{equation}
Also note that $\distSZ(x,y)<\infty$ for all $x,y\in\Z^d$ if and only if $\set_\infty$ is connected. 

The following lemma provides a preliminary bound on the shape of balls in $\Z^d$ with respect to $\distSZ$. 
For $x\in\Z^d$ and $r\geq 0$, let $\ballSZ(x,r)$ be the ball in $\Z^d$ with center at $x$ and radius $\lfloor r\rfloor$ with respect to $\distSZ$, i.e., 
\begin{equation}\label{def:ballSZ}
\ballSZ(x,r) = \left\{y\in\Z^d~:~\distSZ(x,y)\leq r\right\} .\
\end{equation}
\begin{lemma}\label{l:chemdistball:crudebound}
Assume that the family of probability measures $\mathbb P^u$, $u\in(a,b)$, satisfies \p{} -- \ppp{} and \s{} -- \sss{}. 
For any $u\in(a,b)$, there exists $\chemconstt=\chemconstt(u)<\infty$ such that 
for all $x\in\Z^d$ and $R\geq 1$,  
\begin{equation}\label{eq:chemdistball:crudebound}
\mathbb P^u\left[\ballZ(x,R)\subseteq \ballSZ(x,\chemconstt R)\right]\geq
 1 - \chemconstt e^{-(\binlog R)^{1+\constS}/\chemconstt} ~,\
\end{equation}
where $\constS = \constS(u)>0$ is defined in \eqref{eq:funcS}.
\end{lemma}
\begin{proof}[Proof of Lemma~\ref{l:chemdistball:crudebound}]
By \p{}, \eqref{translation_inv_psi}, and \eqref{def:distSZ}, it suffices to prove the lemma for $x = 0\in\Z^d$. 
Fix $u\in(a,b)$. 
By \eqref{eq:C1:infty}, there exist $c = c(u)>0$ and $C = C(u)<\infty$ such that for all $y\in\Z^d$ and $R\geq 1$, 
\begin{equation}\label{psi_y_close_to_y}
\P^u \left[ \, |\Phi(y, \set_\infty )-y|_\infty > R \, \right] \leq C e^{-c\funcS(u,R)}.
\end{equation}
Together with Theorem~\ref{thm:Sinfty:chemdist} and \eqref{eq:funcS}, \eqref{psi_y_close_to_y} implies that 
there exists $\chemconstt = \chemconstt(u)<\infty$ such that for all $R\geq 1$, 
\begin{equation}\label{useful_for_referees_expect_remark}
\mathbb P^u\left[\exists \; z\in \ballZ(0,R)~:~\distSZ(0,z) > \chemconstt R\right] \leq \chemconstt e^{-(\binlog R)^{1+\constS}/\chemconstt}~ .\
\end{equation}
This implies \eqref{eq:chemdistball:crudebound} for $x=0$,  hence the proof of Lemma~\ref{l:chemdistball:crudebound} is complete. 
\end{proof}

Now we state a variant of Theorem \ref{thm:shapeThm} in which $\distS$ and 
$\P^u [ \cdot \, \vert \, 0 \in \set_\infty]$ are replaced by $\distSZ$ and $\P^u$, respectively. 

\begin{proposition}  \label{prop:shapeThm_ext}
Assume that the family of probability measures $\mathbb P^u$, $u\in(a,b)$, satisfies \p{} -- \ppp{} and \s{} -- \sss{}. 
For any $u\in(a,b)$, there exists a convex compact subset $D_u \subset \R^d$ 
such that for all $0<\varepsilon<1$ there exists a 
$\P^u$-a.s.\ finite 
random variable $\tilde{\shaperad}_{\varepsilon,u}$ satisfying
\begin{align} \label{eq:shapeThm_ext_1}
 \forall \,
\shaperad \geq \tilde{\shaperad}_{\varepsilon,u} \; : \; \Z^d \cap (1-\varepsilon)\shaperad \cdot D_u 
&\subseteq \ballSZ (0,\shaperad), \\
\label{eq:shapeThm_ext_2}
\forall \,
\shaperad \geq \tilde{\shaperad}_{\varepsilon,u} \; : \; 
\Z^d  \cap (1+\varepsilon)\shaperad \cdot D_u 
&\supseteq \ballSZ (0,\shaperad).
\end{align}
\end{proposition}
Before we prove Proposition \ref{prop:shapeThm_ext} in Section \ref{subsection_proof_of_ext_shape_proposition}, 
we  deduce Theorem \ref{thm:shapeThm} from it.

\begin{proof}[Proof of Theorem \ref{thm:shapeThm}] Let $u\in(a,b)$. We first observe that
\begin{eqnarray*}
\P^u\left[0 \in \set_\infty\right]
&\underset{\eqref{eq:shapeThm_ext_2}}{\stackrel{\eqref{eq:shapeThm_ext_1}}{=}}
&\P^u\left[0 \in \set_\infty, \;  
 \forall \,
\shaperad \geq \tilde{\shaperad}_{\varepsilon,u} \; : \; 
\begin{array}{c}
\set_\infty \cap (1-\varepsilon)\shaperad \cdot D_u \subseteq 
\set_\infty \cap \ballSZ (0,\shaperad)\\   
\set_\infty \cap (1+\varepsilon)\shaperad \cdot D_u \supseteq
\set_\infty \cap \ballSZ (0,\shaperad)
\end{array}
\right]\\
&\stackrel{\eqref{chemdist_restriction}}{=}
&\P^u\left[0 \in \set_\infty, \;  
 \forall \,
\shaperad \geq \tilde{\shaperad}_{\varepsilon,u} \; : \; 
\begin{array}{c}
\set_\infty \cap  (1-\varepsilon)\shaperad \cdot D_u \subseteq \ballS (0,\shaperad)\\
\set_\infty \cap  (1+\varepsilon)\shaperad \cdot D_u \supseteq \ballS (0,\shaperad)
\end{array}
 \right].
\end{eqnarray*}
Now the statement of Theorem \ref{thm:shapeThm} follows 
(with the same $D_u$ and  $\tilde{\shaperad}_{\varepsilon,u}$ as in Proposition \ref{prop:shapeThm_ext})
if we divide the above equality by 
$\P^u\left[0 \in \set_\infty\right]$ (which is positive by \sss{}).
\end{proof}

\subsection{Proof of Proposition \ref{prop:shapeThm_ext}}\label{subsection_proof_of_ext_shape_proposition}

\begin{lemma}\label{lemma_norm} 
Assume that the family of probability measures $\mathbb P^u$, $u\in(a,b)$, satisfies \p{} -- \ppp{} and \s{} -- \sss{}.
For any $u\in(a,b)$, there exists a unique norm $p_u(\cdot)$ on $\R^d$ 
such that for any $x \in \Z^d$,
\begin{equation}\label{subergodic_limit}
p_u(x)= \lim_{n \to \infty} \frac{1}{n} \distSZ(0,n  x)\qquad \mathbb P^u\text{-a.s.}
\end{equation}
\end{lemma}

\begin{proof}
The proof is similar to the analogous result about the chemical distance on the unique infinite cluster of
supercritical Bernoulli percolation in \cite[Section 5, Remark (g)]{GM90}.
We use \cite[Theorem 1]{kingman} to deduce \eqref{subergodic_limit}.
In accordance with the notation of \cite{kingman}, we fix $u \in (a,b)$ and $x \in \Z^d$ and define 
\[
x_{m,n}:=\distSZ (m  x,\, n  x), \quad m \leq n \in \N.
\] 
Now we check that  the family $(x_{m,n} \, : \; m \leq n \in \N)$ of 
random variables is a \emph{subadditive process} for $\mathbb P^u$, i.e., it satisfies
\cite[Conditions $\mathbf{S}_1,\mathbf{S}_2,\mathbf{S}_3$]{kingman}.
In our setting, these conditions correspond to 
\eqref{triangle_subadditive}, \eqref{distances_translation_inv}, and \eqref{expected_chem_dist_finite} below.

For $l\leq m \leq n \in \N$ the subadditivity condition 
 \begin{equation}\label{triangle_subadditive}
 x_{l,n} \leq x_{l,m} + x_{m,n}
\end{equation}
follows from the triangle inequality for $\distSZ(\cdot,\cdot)$.

By \p{} and \eqref{translation_inv_psi}, the joint law of
$\distSZ$-distances of the vertices of $\Z^d$ under $\P^u$ is translation invariant, 
in particular, 
\begin{equation}\label{distances_translation_inv}
\text{ the joint distribution of $(x_{m+1,n+1} \, : \; m \leq n \in \N)$ is
the same as that of $(x_{m,n} \, : \; m \leq n \in \N)$,}
\end{equation}
moreover, the above shift is ergodic by \p{}.
For all $n \in \N$ we have
\begin{equation}\label{expected_chem_dist_finite}
0 \leq \mathbb E^u [ x_{0,n} ] = \sum_{r=0}^{\infty} \mathbb P^u [\distSZ (0,\, n  x)>r ]
 \stackrel{ \eqref{useful_for_referees_expect_remark} }{<} \infty.
\end{equation}
Having checked that
$(x_{m,n} \, : \; m \leq n \in \N)$  is a subadditive process for $\mathbb P^u$, 
we can use \cite[Theorem 1]{kingman} to obtain \eqref{subergodic_limit} with
\begin{equation}\label{ergodic_limit_inf}
 p_u(x) =
\inf_{n \in \N}  \frac{1}{n}~ \mathbb E^u[ \distSZ(0,n x) ]
=
\lim_{n \to \infty} \frac{1}{n}~ \mathbb E^u[ \distSZ(0,n x) ].
\end{equation}

Now we extend $p_u(\cdot)$ from $\Z^d$ to $\R^d$, so that $p_u(\cdot)$ becomes a norm on $\R^d$.
We first observe that the identity $p_u(n  x)=n  p_u(x)$ follows from \eqref{subergodic_limit}  
 for all $x \in \Z^d$ and $n \in \N$. This allows us to consistently extend $p_u(\cdot)$ to 
 $\mathbb{Q}^d$ by letting $p_u(x)=\frac{1}{n} p_u( n  x)$ for any $x \in \mathbb{Q}^d$,
where $n \in \N$ is chosen so that $n x \in \Z^d$.

By \p{},
 the triangle inequality for $\distSZ(\cdot,\cdot)$,  and \eqref{ergodic_limit_inf},
we obtain 
\begin{equation}\label{norm_properties_opposite_subadd}
 p_u(-x)=p_u(x),
 \qquad p_u(x+y) \leq p_u(x)+p_u(y), \qquad    x,y \in \mathbb{Q}^d.
\end{equation}
By \eqref{psi_y_close_to_y}, \eqref{chemdist_restriction}, \eqref{expected_chem_dist_finite},
 \eqref{ergodic_limit_inf}, and \eqref{norm_properties_opposite_subadd}, we have 
\begin{equation}\label{norm_properties_bounded_above_below}
 |x|_1 \leq p_u(x) \leq \max_{1\leq i \leq d}  \mathbb{E}[ \distSZ(0,e_i) ]   \cdot  |x|_1,
\qquad x \in \mathbb{Q}^d.
\end{equation}
Thus 
 $p_u(\cdot)$ is a norm on $\mathbb{Q}^d$, which admits a unique
 continuous extension to $\R^d$.
\end{proof}

In the sequel we denote by $\ballR(x,r)$ the closed $L_{\infty}$-ball in $\R^d$ with
 radius $r \in \R_+$ and center $x \in \R^d$.

Note that with the above definitions we have 
$\ballZ(x,r)= \Z^d \cap \ballR(x,r)$ for any $x \in \Z^d$ and $r \in \R_+$.

\begin{proof}[Proof of Proposition \ref{prop:shapeThm_ext}]
Let $u\in(a,b)$. We define the compact convex set 
\begin{equation}\label{def_eq_shape_set}
 D_u := \{x \in \R^d : p_u(x) \le 1\}
 \end{equation}
using the norm $p_u(\cdot)$ from Lemma \ref{lemma_norm}.
We will show that \eqref{eq:shapeThm_ext_1} and \eqref{eq:shapeThm_ext_2} hold with this choice of $D_u$.

\medskip

We first show the inclusion \eqref{eq:shapeThm_ext_1}. 
Recall the constant $\chemconstt$ appearing in \eqref{eq:chemdistball:crudebound}.
For every $0<\varepsilon<1$ we choose $k=k(\varepsilon) \in \N$ and
$x_1, \ldots, x_k \in \mathbb Q^d \cap (1-\varepsilon) D_u$ such that
\begin{equation} \label{eq:DuCover}
 (1-\varepsilon) D_u \subseteq \bigcup_{j=1}^k \ballR \left(x_j,\frac{\varepsilon}{3 \chemconstt}\right).
\end{equation}
We fix $n \in \N$ such that $n  x_j \in \Z^d$ for all $1 \le j \le k$ and define
\begin{equation} \label{def_eq_x_L_j}
x_j^\shaperad = \left\lfloor \frac{\shaperad}{n} \right\rfloor \cdot n  x_j \in \Z^d, \quad
 1 \leq j \leq k, \quad R \in \N. 
 \end{equation}
Note that we have $\frac{1}{R} x^\shaperad_j \to x_j$ as $\shaperad \to \infty$ for every
$1\leq j\leq k$, thus  there is a constant $C(\varepsilon)$ such that
\begin{equation}\label{trivi_approx_rationals}
 \forall \,
\shaperad \geq C(\varepsilon), \; 1 \leq j \leq k 
 \; : \; 
 \Z^d \cap \ballR \left(\shaperad x_j,\frac{\varepsilon}{3 \chemconstt} \shaperad \right)
\subseteq 
\ballZ \left(x^\shaperad_j,\frac{\varepsilon}{2 \chemconstt} \shaperad \right).
\end{equation}

From Lemma \ref{lemma_norm} we obtain
\begin{equation*}
\forall \; 1 \leq j \leq k  \; : \;
 \lim_{\shaperad \to \infty} \frac{1}{\shaperad} \distSZ \left(0,\,  x^\shaperad_j   \right)
\stackrel{\mathbb P^u\text{-a.s.}}{=} \frac{1}{n}p_u(nx_j) =  p_u(x_j) 
\stackrel{\eqref{def_eq_shape_set}}{\leq} 1-\varepsilon.
\end{equation*}
Thus, there exists a 
$\P^u$-almost surely finite random variable $\tilde{\shaperad}^1_{\varepsilon,u}$ such that (recall \eqref{def:ballSZ})
\begin{equation} \label{eq:xjDist}
  \forall \,
\shaperad \geq \tilde{\shaperad}^1_{\varepsilon,u}, \; 1 \leq j \leq k 
 \; : \; 
  x^\shaperad_j \;  \in \;
 \ballSZ \left(0,(1-\frac{\varepsilon}{2})\shaperad \right).
\end{equation}
By \eqref{eq:chemdistball:crudebound} and the Borel-Cantelli lemma,
there exists a 
$\P^u$-almost surely finite random variable $\tilde{\shaperad}^2_{\varepsilon,u}$ such that
\begin{equation}\label{eq:borel_cantelli_balls}
 \forall \,
\shaperad \geq \tilde{\shaperad}^2_{\varepsilon,u}, \; 1 \leq j \leq k 
 \; : \; 
 \ballZ \left(   x^\shaperad_j, \, \frac{\varepsilon}{2 \chemconstt} \shaperad \right) \subseteq
\ballSZ \left(   x^\shaperad_j, \, \frac{\varepsilon}{2} \shaperad \right).
\end{equation}
 Now we are ready to conclude the proof of \eqref{eq:shapeThm_ext_1}. We have that $\mathbb P^u$-almost surely, 
\begin{multline}\label{eq:proof_of_shape_ext_1}
 \forall \,
\shaperad \geq \max \{\tilde{\shaperad}^1_{\varepsilon,u}, \tilde{\shaperad}^2_{\varepsilon,u}, C(\varepsilon) \}
 \; : \; 
\Z^d \cap (1-\varepsilon)\shaperad \cdot D_u 
\stackrel{\eqref{eq:DuCover}}{\subseteq}
\Z^d \cap \bigcup_{j=1}^k \ballR \left(\shaperad x_j,\frac{\varepsilon}{3 \chemconstt} \shaperad \right) \\
\stackrel{\eqref{trivi_approx_rationals}}{\subseteq} 
\bigcup_{j=1}^k
 \ballZ \left(  x^\shaperad_j, \, \frac{\varepsilon}{2 \chemconstt} \shaperad \right)
\stackrel{\eqref{eq:borel_cantelli_balls}}{\subseteq} 
\bigcup_{j=1}^k \ballSZ\left(   x^\shaperad_j, \, \frac{\varepsilon}{2} \shaperad \right)
\stackrel{\eqref{eq:xjDist}}{\subseteq}  \ballSZ (0,\shaperad).
\end{multline}

It remains to show the inclusion \eqref{eq:shapeThm_ext_2}.
We first note that by \eqref{norm_properties_bounded_above_below} and \eqref{def_eq_shape_set}, 
\[
D_u \subseteq \ballR(0,1).
\]
For every $0<\varepsilon<1$,  there exists $k \in \N$ and
 $x_1, \ldots, x_k \in \mathbb Q^d \cap \ballR(0,2) \setminus (1+\varepsilon) D_u$ such that
\begin{equation}\label{eq:DuCover_ketto}
\ballR(0,2) \setminus (1+\varepsilon) D_u \subseteq \bigcup_{j=1}^k 
\ballR \left(x_j,\frac{\varepsilon}{3 \chemconstt}\right).
\end{equation}
Let us fix $n \in \N$ such that $n  x_j \in \Z^d$ for all $1 \le j \le k$. 
Define $x_j^\shaperad$ by the formula \eqref{def_eq_x_L_j}.
Similarly to \eqref{trivi_approx_rationals}, \eqref{eq:xjDist} and \eqref{eq:borel_cantelli_balls} 
there exists a constant $C'(\varepsilon)< \infty$ and 
$\P$-almost surely finite random variables $\tilde{\shaperad}^3_{\varepsilon,u}$, 
$\tilde{\shaperad}^4_{\varepsilon,u}$ such that
\begin{equation}\label{trivi_approx_rationals_ketto}
 \forall \,
\shaperad \geq C'(\varepsilon), \; 1 \leq j \leq k 
 \; : \; 
 \Z^d \cap \ballR \left(\shaperad x_j,\frac{\varepsilon}{3 \chemconstt} \shaperad \right)
\subseteq 
\ballZ \left(x^\shaperad_j,\frac{\varepsilon}{2 \chemconstt} \shaperad \right),
\end{equation}
\begin{equation} \label{eq:xjDist_ketto}
  \forall \,
\shaperad \geq \tilde{\shaperad}^3_{\varepsilon,u}, \; 1 \leq j \leq k 
 \; : \; 
  x^\shaperad_j \; \in \; 
\Z^d \setminus \ballSZ\left(0,(1+\frac{\varepsilon}{2})\shaperad \right),
\end{equation}
\begin{equation}\label{eq:borel_cantelli_balls_ketto}
 \forall \,
\shaperad \geq \tilde{\shaperad}^4_{\varepsilon,u}, \; 1 \leq j \leq k 
 \; : \; 
 \ballZ \left(   x^\shaperad_j, \, \frac{\varepsilon}{2 \chemconstt} \shaperad \right) \subseteq
\ballSZ \left(   x^\shaperad_j, \, \frac{\varepsilon}{2} \shaperad \right).
\end{equation}
Similarly to \eqref{eq:proof_of_shape_ext_1}, we obtain that $\mathbb P^u$-almost surely, 
\begin{multline}\label{complements_inclusion}
\forall \,
\shaperad \geq \max \{\tilde{\shaperad}^3_{\varepsilon,u}, \tilde{\shaperad}^4_{\varepsilon,u}, C'(\varepsilon) \}
 \; : \; 
  \ballZ(0,2\shaperad) \setminus (1+\varepsilon)\shaperad \cdot D_u   
\stackrel{ \eqref{eq:DuCover_ketto} }{\subseteq} 
\Z^d \cap \bigcup_{j=1}^k \ballR \left(\shaperad x_j,\frac{\varepsilon}{3 \chemconstt} \shaperad \right)
\stackrel{\eqref{trivi_approx_rationals_ketto}}{\subseteq}
\\
 \bigcup_{j=1}^k
\ballZ \left(  x^\shaperad_j, \, \frac{\varepsilon}{2 \chemconstt} \shaperad \right)
\stackrel{\eqref{eq:borel_cantelli_balls_ketto}}{\subseteq} 
\bigcup_{j=1}^k \ballSZ \left(   x^\shaperad_j, \, \frac{\varepsilon}{2} \shaperad \right)
\stackrel{\eqref{eq:xjDist_ketto}}{\subseteq} \Z^d \setminus 
\ballSZ (0,\shaperad).
\end{multline}

By  \eqref{psi_y_close_to_y}, \eqref{chemdist_restriction},
 and the Borel-Cantelli lemma, we obtain that
there exists a 
$\P^u$-a.s.\ finite random variable $\tilde{\shaperad}^5_{\varepsilon,u}$ such that
\[
\forall \,
\shaperad \geq \tilde{\shaperad}^5_{\varepsilon,u} \; : \; 
\ballSZ(0, \shaperad) \subseteq \ballZ(0, 2\shaperad).
\] 
If we combine this observation with \eqref{complements_inclusion}, we obtain 
\eqref{eq:shapeThm_ext_2}, which concludes the proof of  Proposition \ref{prop:shapeThm_ext}
and Theorem \ref{thm:shapeThm}.
\end{proof}

\paragraph{Acknowledgements.}
A large part of the work on this paper was done while the authors were at ETH Z\"urich. 
During that time, the research of AD was supported by an ETH Fellowship, and  
the research of BR and AS was supported by the grant ERC-2009-AdG 245728-RWPERCRI.
We thank an anonymous referee for reading and commenting on the manuscript.

\end{document}